\documentclass[EJP,preprint]{ejpecp}

\usepackage[T1]{fontenc}
\usepackage[utf8]{inputenc}

\usepackage{enumerate}
\usepackage{physics}

\SHORTTITLE{Wasserstein Contraction and Spectral Gap of Slice Sampling Revisited}
\TITLE{Wasserstein Contraction and Spectral Gap\\of Slice Sampling Revisited}
\AUTHORS{Philip Schär\footnote{Friedrich Schiller University Jena, Germany. \EMAIL{philip.schaer@uni-jena.de}}}

\AMSSUBJ{65C05; 60J05; 60J22}
\SUBMITTED{May 26, 2023}
\ACCEPTED{September 25, 2023}

\VOLUME{0}
\YEAR{2020}
\PAPERNUM{0}
\DOI{10.1214/YY-TN}

\ABSTRACT{We propose a new class of Markov chain Monte Carlo methods, called \textit{$k$-polar slice sampling} ($k$-PSS), as a technical tool that interpolates between and extrapolates beyond uniform and polar slice sampling. By examining Wasserstein contraction rates and spectral gaps of $k$-PSS, we obtain strong quantitative results regarding its performance for different kinds of target distributions. Because $k$-PSS contains uniform and polar slice sampling as special cases, our results significantly advance the theoretical understanding of both of these methods.
In particular, we prove realistic estimates of the convergence rates of uniform slice sampling for arbitrary multivariate Gaussian distributions on the one hand, and near-arbitrary multivariate $t$-distributions on the other. Furthermore, our results suggest that for heavy-tailed distributions, polar slice sampling performs dimension-independently well, whereas uniform slice sampling suffers a rather strong curse of dimensionality.}

\newcommand{\ra}{\rightarrow}
\newcommand{\RA}{\Rightarrow}
\newcommand{\LRA}{\Leftrightarrow}
\newcommand{\N}{\mathbb{N}}
\newcommand{\R}{\mathbb{R}}
\newcommand{\sph}{\mathbb{S}}
\renewcommand{\d}{\text{d}}
\newcommand{\epsi}{\varepsilon}
\newcommand{\I}{\mathbb{I}}
\newcommand{\ind}{\mathds{1}}
\newcommand{\supp}{\,\textsf{supp}}
\newcommand{\Pc}{\mathcal{P}}
\newcommand{\B}{\mathcal{B}}
\newcommand{\Bf}{\mathfrak{B}}
\newcommand{\G}{\mathcal{G}}
\newcommand{\D}{\mathcal{D}}
\newcommand{\Nc}{\mathcal{N}}
\newcommand{\U}{\mathcal{U}}

\newcommand{\gap}{\textsf{gap}}
\newcommand{\W}{\mathcal{W}}
\newcommand{\Dob}{\textsf{Dob}}
\newcommand{\IAT}{\textsf{IAT}}
\newcommand{\dif}[2]{\frac{\d #1}{\d #2}}

\newcommand\ccint[2]{\,[#1,#2]}
\newcommand\ooint[2]{\,]#1,#2[}
\newcommand\ocint[2]{\,]#1,#2]}
\newcommand\coint[2]{\,[#1,#2[}

\newcommand\oointv[2]{\,\left]#1,#2\right[}
\newcommand\ocintv[2]{\,\left]#1,#2\right]}
\newcommand\cointv[2]{\,\left[#1,#2\right[}

\newtheorem{discussion}[theorem]{Discussion}

\begin{document}

\section{Introduction} \label{Sec:intro}

Sampling from intractable distributions is a ubiquitous problem in all fields that rely on probabilistic modeling. One of the standard approaches to solve this problem is Markov chain Monte Carlo (MCMC), which works by constructing a Markov chain whose iterate distribution converges to the target distribution as the chain progresses. Samples from such a chain are then used as approximate samples from the target. Although there are many heuristics for assessing the quality of such samples, there is no general way to definitively determine how well they are suited for estimating an unknown target quantity. It is therefore of vital importance to gain a good understanding of the theoretical properties of MCMC methods, as these provide more reliable, general insights than heuristics and numerical results.

In this paper, we focus on a class of MCMC methods called \textit{slice sampling}, which originated in physics research and was popularized in statistics in \cite{BesagGreen}. We view slice sampling as being defined for distributions on $(\R^d, \B(\R^d))$, where $\B(G)$ denotes the Borel-sigma algebra of a set $G$, though we note that recent advances have also extended the concept to the $d$-sphere \cite{SphericalSS}. More specifically, the setting we assume is the following. Let $\nu$ be a probability measure (in the following often just called \textit{distribution}) on $(\R^d, \B(\R^d))$. Suppose that $\nu$ has a potentially unnormalized Lebesgue density, that is, a measurable function $\eta: \R^d \ra \R_+ := \coint{0}{\infty}$ such that
\begin{equation*}
	\nu(A)
	= \frac{\int_A \eta(x) \d x}{\int_{\R^d} \eta(x) \d x} ,
	\qquad A \in \B(\R^d) .
\end{equation*}
We assume that we can evaluate $\eta$ at any given point $x \in \R^d$ but that its normalization constant (the denominator in the above expression) is unknown. Slice sampling for $\nu$ further assumes $\eta$ to be factorized into measurable functions $\eta_i: \R^d \ra \R_+$, $i=0,1$, such that
\begin{equation}
	\eta(x) = \eta_0(x) \eta_1(x) , \qquad x \in \R^d .
	\label{Eq:ss_factorization}
\end{equation}
Given an initial state $x_0 \in \R^d$, slice sampling for $\nu$ then generates a realization $(x_n)_{n \in \N_0}$ of a Markov chain $(X_n)_{n \in \N_0}$ with invariant distribution $\nu$ by setting $X_0 := x_0$ and using the following steps to perform the transition from $X_{n-1}$ to $X_n$:
\begin{enumerate}
	\item Sample $T_n \sim \U(\ooint{0}{\eta_1(x_{n-1})})$, call the result $t_n$.
	\item Sample $X_n \sim \mu_{t_n}$, where $\mu_{t_n}$ is defined by\footnote{We denote by $\I(cond)$ an indicator that takes the value $1$ whenever the condition $cond$ is satisfied and the value $0$ otherwise.}.
	\begin{equation*}
		\mu_t(A)
		:= \frac{\int_A \eta_0(x) \I(\eta_1(x) > t) \d x}{\int_{\R^d} \eta_0(x) \I(\eta_1(x) > t) \d x} ,
		\qquad A \in \B(\R^d), t > 0,
	\end{equation*}
	call the result $x_n$.
\end{enumerate}
Following \cite{GPSS}, we refer to step 2 as the \textit{$X$-update} of the slice sampling transition. Intuitively, the $X$-update samples the new state from the distribution whose (unnormalized) density is the first factor $\eta_0$ restricted to the super-level set (or \textit{slice}) w.r.t.~the second factor $\eta_1$ at the current threshold $t_n$.

The above definition is frequently restricted by specifying a mapping that determines for each target density $\eta$ the factorization \eqref{Eq:ss_factorization}, thereby defining what we call a \textit{slice sampling variant}. In our view, only three such variants are currently practically relevant. First there is \textit{uniform slice sampling} (USS), which uses the factorization
\begin{equation*}
	\eta_0(x) := 1 , 
	\qquad 
	\eta_1(x) := \eta(x) .
\end{equation*}
Then there is \textit{polar slice sampling} (PSS), proposed in \cite{PolarSS}, which uses
\begin{equation*}
	\eta_0(x) := \norm{x}^{1-d} ,
	\qquad 
	\eta_1(x) := \norm{x}^{d-1} \eta(x)
\end{equation*}
for $x \neq 0$, where $\norm{\cdot}$ denotes the Euclidean norm on $\R^d$.
Finally, there is \textit{Gaussian slice sampling} (GSS), which sets $\eta_0$ to the density of a multivariate Gaussian and $\eta_1$ to $\eta$ divided by it.

As all three variants lead to an $X$-update that is typically infeasible to implement, in practice one uses methods that mimic these ``ideal'' formulations while enabling computationally efficient implementation. For USS, one can use the stepping-out and shrinkage procedures proposed in \cite{SSNeal}, for PSS the Gibbsian polar slice sampling framework recently developed in \cite{GPSS}, and for GSS the elliptical slice sampling framework proposed in \cite{EllipticalSS}. Of course theoretical results for the ideal methods do not readily yield results for their efficient approximations. However, for USS it has been shown in \cite{Latuszynski} that some theoretical results, even quantitative ones, can be explicitly transferred from the ideal to the efficient version. We are cautiously optimistic that analogous theories can also be developed for PSS and GSS, making it all the more worthwhile to investigate the theoretical properties of the ideal methods.

Though we are of course also interested in GSS, we focus our efforts in this paper exclusively on the other two methods, USS and PSS. Moreover, we observe that these two can often be analyzed simultaneously. To this end, we propose a new class of slice sampling variants, which we term \textit{$k$-polar slice sampling} ($k$-PSS). For any fixed $k \in \R_{>0} := \ooint{0}{\infty}$, we define $k$-PSS as slice sampling with factorization
\begin{align*}
	\eta_0(x) 
	&:= \eta_{k,0}(x) 
	:= \norm{x}^{k-d} , \\
	\eta_1(x)
	&:= \eta_{k,1}(x) 
	:= \norm{x}^{d-k} \eta(x) 
\end{align*}
for $x \neq 0$. Clearly, $k$-PSS coincides with PSS for $k=1$ and with USS for $k=d$. For $1 \leq k \leq d$, it continuously interpolates between USS and PSS, and for $k < 1$ and $k > d$ it extrapolates beyond them in both directions. As we will see in the following, the quantitative theoretical properties of $k$-PSS often also interpolate between those of USS and PSS. However, we emphasize right away that we do not deem the $k$-PSS methods apart from USS and PSS to be useful in actual practical use as samplers. Rather, we view them as technical tools that we can use to prove theoretical results, and, as we shall see later on (in Theorem \ref{Thm:gap_estimate}), their use in this regard goes far beyond unifying the analysis of USS and PSS.

Let us now briefly mention some related work. It is well-known that slice sampling is geometrically ergodic (essentially meaning that the total variation distance between iterate distribution and target distribution converges geometrically to zero) under extremely weak assumptions \cite{SliceConv}. However, this is a purely qualitative result, as the convergence can in practice still be arbitrarily slow, and -- unsurprisingly -- the known quantitative results are much less extensive. Some quantitative convergence guarantees were already established in \cite{SliceConv}, but only under fairly restrictive conditions, including on the initial value, and the guarantees given there were, in retrospect, far from sharp. Realistic quantitative results on slice sampling are a comparatively recent achievement. In \cite{SlavaUSS}, the authors examine Wasserstein contraction rates and spectral gaps (see below) of USS and in \cite{PSS_paper} some of their methodology is extended to general slice sampling and a dimension-independent lower bound on the spectral gap of PSS, for rotationally invariant targets that are log-concave along rays emanating from the origin, is proven.

Following the example of \cite{SlavaUSS} and \cite{PSS_paper}, we also focus on Wasserstein contraction rates and spectral gaps as quantities of interest in order to arrive at realistic quantitative results. We now briefly introduce and motivate both of these quantities.

Let $(G,\G)$ be a measurable space and let $P$ be a Markov kernel on $G \times \G$ with invariant distribution $\nu \in \Pc(G)$, where $\Pc(G)$ denotes the set of probability measures on $(G,\G)$. Define the $L_2$-space belonging to $\nu$ as
\begin{equation*}
	L_2(\nu)
	:= \{ g: G \ra \R \;\colon \norm{g}_{2,\nu} < \infty \} ,
\end{equation*}
where
\begin{equation*}
	\norm{g}_{2,\nu} := \left( \int_G g(x)^2 \nu(\d x) \right)^{1/2} .
\end{equation*}
Observe that $P$ acting on measurable functions $g: G \ra \R$ via $g \mapsto P g$ with
\begin{equation*}
	P g (x)
	:= \int_G g(y) P(x, \d y) ,
	\qquad x \in G
\end{equation*}
defines a linear operator mapping $L_2(\nu) \ra L_2(\nu)$. Note further that by interpreting $\nu$ as a Markov kernel that is constant in its first argument, $\nu$ also defines a linear operator $L_2(\nu) \ra L_2(\nu)$. This framing allows us to define the \textit{spectral gap} of $P$ as
\begin{equation*}
	\gap_{\nu}(P)
	:= 1 - \norm{P - \nu}_{L_2(\nu) \ra L_2(\nu)} ,
\end{equation*}
where $\norm{\cdot}_{L_2(\nu) \ra L_2(\nu)}$ is the operator norm w.r.t.~$\norm{\cdot}_{2,\nu}$. 

The result regarding spectral gaps that one is typically interested in is a lower bound, i.e.~a statement of the form $\gap_{\nu}(P) \geq \delta$ for some $\delta > 0$ (in the following simply called \textit{spectral gap estimate}). The reason for this is that there are various useful implications of such a bound. Suppose $\gap_{\nu}(P) \geq 1 - \rho$ for some $\rho \in \ooint{0}{1}$, then it is well-known, see e.g.~\cite[Lemma 2]{Novak}, that for any initial distribution $\xi \in \Pc(G)$ with $\xi \ll \nu$ one has
\begin{equation}
	d_{\text{tv}}(\xi P^n, \nu)
	\leq \rho^n \norm{ \dif{\xi}{\nu} - 1 }_{2,\nu} , 
	\qquad n \in \N_0 ,
	\label{Eq:gap_implied_convergence}
\end{equation}
where $\dif{\xi}{\nu}$ denotes the Radon-Nikodym derivative of $\xi$ w.r.t.~$\nu$ and $d_{\text{tv}}$ the total variation distance, which for $\xi_1, \xi_2 \in \Pc(G)$ is defined as
\begin{equation*}
	d_{\text{tv}}(\xi_1, \xi_2)
	:= \sup_{A \in \G} \abs{\xi_1(A) - \xi_2(A)} .
\end{equation*}
Note that $\xi P^n$ is the distribution of $X_n$ when $(X_n)_{n \in \N_0}$ is a Markov chain with transition kernel $P$ and initial distribution $X_0 \sim \xi$, see \cite[Theorem 1.3.4]{MCbook}. Aside from this geometric convergence result, consequences of a spectral gap estimate include a quantitative bound on the error of Monte Carlo integration \cite[Theorem 3.41]{RudolfDiss}, a central limit theorem \cite{gapCLT} and an estimate of the central limit theorem's asymptotic variance \cite{gapCLTasvar}.

The second kind of quantity we examine is the Wasserstein contraction rate. To express it, denote by $\W$ the \textit{Wasserstein 1-distance}, which for $\xi_1, \xi_2 \in \Pc(\R^d)$ is defined as
\begin{equation*}
	\W(\xi_1, \xi_2)
	:= \inf_{\gamma \in \Gamma(\xi_1, \xi_2)} \int_{\R^d \times \R^d} \norm{x - y} \gamma(\d x, \d y) ,
\end{equation*}
where $\Gamma(\xi_1, \xi_2)$ denotes the set of couplings of $\xi_1$ and $\xi_2$,
\begin{equation*}
	\Gamma(\xi_1, \xi_2)
	:= \{ \gamma \in \Pc(\R^d \times \R^d) \mid \forall A \in \B(\R^d): \; \gamma(A \times \R^d) = \xi_1(A), \; \gamma(\R^d \times A) = \xi_2(A) \} .
\end{equation*}
Suppose now that $G \subseteq \R^d$ is a non-empty, closed set and let $P$ be a Markov kernel on $G \times \B(G)$. Then the \textit{Dobrushin coefficient} $\Dob(P)$ of $P$ is given by
\begin{equation*}
	\Dob(P)
	= \sup_{x,y \in G, \; x \neq y} \frac{\W(P(x,\cdot),P(y,\cdot))}{\norm{x - y}} ,
\end{equation*}
see \cite[Definition 20.3.1, Lemma 20.3.2]{MCbook}. Given this definition, $P$ is called \textit{Wasserstein contractive} if $\Dob(P) < 1$. Moreover, it is called \textit{Wasserstein contractive with rate $\rho$} if $\Dob(P) \leq \rho < 1$.

There are two main reasons why Wasserstein contraction is an extremely desirable property. The first is that it implies a very intuitive convergence result w.r.t.~the Wasserstein distance: Suppose that $P$ is Wasserstein contractive with rate $\rho$ and has $\nu \in \Pc(G)$ as an invariant distribution, then for any initial distribution $\xi \in \Pc(G)$ one has
\begin{equation*}
	\W(\xi P^n, \nu) \leq \rho^n \W(\xi, \nu) ,
	\qquad n \in \N_0 ,
\end{equation*}
see \cite[Theorem 20.3.4]{MCbook}. In other words, one has a bound for the error expressed as the Wasserstein distance between iterate and target distribution, and that error bound decreases by the factor $\rho$ with each step of the Markov chain, with its coefficient explicitly given by $\W(\xi,\nu)$, the Wasserstein distance between initial and target distribution.

The second reason why Wasserstein contraction is particularly desirable is that it is a strictly more powerful result than a spectral gap estimate. That is, if $P$ is reversible w.r.t.~$\nu$, and if additionally $\nu$ has finite second moment, meaning $\int_{\R^d} \norm{x}^2 \nu(\d x) < \infty$, then one has
\begin{equation}
	\Dob(P) \leq \rho
	\qquad \RA \qquad 
	\gap_{\nu}(P) \geq 1 - \rho ,
	\label{Eq:contraction_to_gap}
\end{equation}
for any $\rho < 1$, see \cite[Proposition 30]{Ricci}. Hence all the implications of a spectral gap estimate that we mentioned above are also implications of a kernel being Wasserstein contractive.

Given that Wasserstein contraction is a strictly more powerful result than a spectral gap, it is unsurprising that the former is also much harder to establish than the latter. Accordingly, we begin our investigation by proving Wasserstein contraction of slice sampling wherever feasible. First, we generalize the contraction result of \cite{SlavaUSS} from USS to $k$-PSS, proving that $k$-PSS is Wasserstein contractive for all target densities $\eta$ for which the factors $\eta_{k,1}$ are rotationally invariant and log-concave. Afterwards, we investigate the Wasserstein contraction rates of slice sampling for heavy-tailed distributions. Here we end up computing precise rates for, on the one hand, USS applied to standard multivariate $t$-distributions and, on the other hand, $k$-PSS applied to a class of distributions that mimics the aforementioned one, while being much simpler to analyze.

We then augment our contraction results by proving spectral gap estimates for broader classes of distributions, even some that are rotationally asymmetric (unlike any of those for which contraction results on slice sampling are available). To enable more elegant and often also further-reaching results, we begin this part of the paper by improving a theorem that can provide realistic estimates of the spectral gap of slice sampling. We note here that the best previously known version of this theorem was proven in \cite{PSS_paper}, whereas the original theorem, applying just to USS, was developed in \cite{SlavaUSS}. Given this improved theorem, we first consider $k$-PSS for classes of rotationally invariant target densities $\eta$ with factors $\eta_{k,1}$ of varying degrees of log-concavity, even reaching into log-convexity. We then examine $k$-PSS for a class of rotationally asymmetric targets, which allows us in particular to prove a realistic spectral gap estimate for USS applied to arbitrary multivariate Gaussian distributions. Finally, we consider USS in isolation once more to augment the aforementioned contraction result for standard multivariate $t$-distributions by a spectral gap estimate for arbitrary multivariate $t$-distributions.

We take this opportunity to emphasize that our interest in the quantitative properties of slice sampling for standard distributions (such as Gaussians or multivariate $t$-distributions) does not reflect a desire to actually apply slice sampling to such targets in practice. Even if strong guarantees regarding the performance of slice sampling are available, standard methods to directly generate samples from these tractable targets (i.e.~ways to suitably transform samples from the uniform distribution) would still be substantially more efficient. Rather, we view the standard distributions we often consider in this paper as proxies for intractable target distributions that share some of the standard distributions' properties (e.g.~their tail behavior). We then expect our results on the performance of slice sampling for standard targets to be indicative, to some extent, of their performance for these somehow similar intractable targets. Of course this then further needs to carry over to the performance of the efficiently implementable approximations of the slice samplers we analyze here, see our remarks on this above.

The paper is structured as follows. Section \ref{Sec:contraction} details our investigation of the Wasserstein contraction rates of slice sampling for different classes of target distributions. Section \ref{Sec:gap_estimates} establishes a theorem that allows estimating the spectral gap of slice sampling, and then discusses several applications of said theorem. We conclude each subsection with a detailed discussion of our results. Some auxiliary results are stated and (if necessary) proven in Appendix \ref{App:aux_res}. Finally, Appendix \ref{App:num_supp} gives some numerical results that strongly suggest the sharpness of one of our spectral gap results, for which we were unable to establish this sharpness theoretically.

Throughout the paper, we use a few non-notational conventions: When talking about densities, we mean densities with respect to the Lebesgue measure, unless stated otherwise. We do not require densities to be normalized\footnote{In our view, this admittedly somewhat unconventional nomenclature is justified by our using the term \textit{density}, rather than \textit{probability density function}, as only the latter explicitly implies normalization.} and will instead explicitly emphasize whenever a density is required or constructed to be normalized. Finally, it should be noted that we view a probability distribution and its density as a fixed unit, ignoring the ambiguous normalization and the fact that non-negative functions which coincide with the density almost everywhere are also densities for the same distribution. Hence, whenever we introduce a distribution with a specific density and talk about a sampler for this distribution, we actually mean the sampler for the density we specified.

\section{Wasserstein Contraction Results} \label{Sec:contraction}

\subsection{Targets with Log-Concave Factors}

Our first result establishes, for any $k \in \R_{>0}$, Wasserstein contraction of $k$-PSS for a class of distributions $\D_k(\R^d)$ that is defined as follows. By this we generalize a result of \cite{SlavaUSS} that proved Wasserstein contraction of USS for all rotationally invariant, log-concave target densities.

\begin{definition} \label{Def:D_k_def}
	For any $k \in \R_{>0}$ we define a class $\D_k(\R^d) \subset \Pc(\R^d)$ of distributions by the constraint that each $\nu \in \D_k(\R^d)$ must have a density $\eta$ of the form
	\begin{equation}
		\eta(x)
		= \norm{x}^{k-d} \exp(-\phi(\norm{x})) \ind_{\ooint{0}{\kappa}}(\norm{x}) ,
		\label{Eq:D_k_def}
	\end{equation}
	where $\kappa \in \ocint{0}{\infty}$ and $\phi: \ooint{0}{\kappa} \ra \R$ is strictly increasing and convex.
\end{definition}

We state our Wasserstein contraction result and prove it using auxiliary results that we provide in Appendix \ref{App:aux_res}.

\begin{theorem} \label{Thm:contraction}
	For any $k > 0$ and $\nu \in \D_k(\R^d)$, $k$-PSS for $\nu$ is Wasserstein contractive with rate $\rho = k/(k+1)$.
\end{theorem}
\begin{proof}
	For the entirety of the proof, arbitrarily fix $k > 0$ and $\nu \in \D_k(\R^d)$ with density $\eta$ as in \eqref{Eq:D_k_def}. To estimate the Wasserstein distance between two instances of $k$-PSS currently positioned at different points $x, y \in \R^d$, we construct a coupling between the respective transition kernels. To this end, we begin by analyzing how the slice sampling procedure simplifies for $k$-PSS on $\D_k(\R^d)$. First of all, note that in this setting we have
	\begin{equation}
		\eta_{k,1}(x) 
		= \exp(-\phi(\norm{x})) \ind_{\ooint{0}{\kappa}}(\norm{x})
		\label{Eq:D_k_likelihood_factor}
	\end{equation}
	for $x \neq 0$. Now observe that\footnote{See Lemma \ref{Lem:phi_inverse} for notation.}
	\begin{equation*}
		\sup \eta_{k,1}
		:= \sup_{0 \neq x \in \R^d} \eta_{k,1}(x) 
		= \exp(- \inf \phi) .
	\end{equation*}
	Hence we only need to consider values smaller than $\sup \eta_{k,1}$ for the variable $t$ standing in for the auxiliary variable $T_n$ in the following. Next, observe that
	\begin{equation*}
		\inf \eta_{k,1}
		:= \inf_{x \in \R^d, \; \norm{x} \in \ooint{0}{\kappa}} \eta_{k,1}(x)
		= \begin{cases}
			\exp(-\sup \phi) & \kappa < \infty , \\
			0 & \kappa = \infty .
		\end{cases}
	\end{equation*}
	We now characterize the slices of $k$-PSS for $\nu$. For $t \in \ocint{0}{\inf \eta_{k,1}}$ and $x \in \R^d$, one has
	\begin{align*}
		\eta_{k,1}(x) > t
		\quad \LRA \quad \norm{x} \in \ooint{0}{\kappa} .
	\end{align*}
	For $t \in \ooint{\inf \eta_{k,1}}{\sup \eta_{k,1}}$ and $0 \neq x \in \R^d$, we get\footnote{See Lemma \ref{Lem:phi_inverse} for context on $\phi^{-1}$.}
	\begin{align*}
		\eta_{k,1}(x) > t \quad 
		&\LRA \quad \exp(-\phi(\norm{x})) > t \quad \text{and} \quad \norm{x} < \kappa \\
		&\LRA \quad \phi(\norm{x}) < - \log t \quad \text{and} \quad \norm{x} < \kappa \\
		&\LRA \quad \norm{x} < \phi^{-1}(-\log t) \quad \text{and} \quad \norm{x} < \kappa \\
		&\LRA \quad \norm{x} < \phi^{-1}(-\log t) ,
	\end{align*}
	where the third equivalence relies on $\phi^{-1}$ being strictly increasing and the fourth on it mapping to $D_{\phi} = \oointv{0}{\kappa}$, s.t.~in particular $\phi^{-1} < \kappa$.
	Using the notation introduced in Lemma \ref{Lem:non_expansion_prop}, we can summarize these two cases as
	\begin{equation}
		\eta_{k,1}(x) > t
		\quad \LRA \quad 
		\norm{x} \in \, \ooint{0}{\widehat{\phi}^{-1}(-\log t)}
		\label{Eq:L(t)_for_contraction}
	\end{equation}
	for any $t \in \ooint{0}{\sup \eta_{k,1}}$. With this we can write
	\begin{equation}
		\eta_{k,0}(x) \I(\eta_{k,1}(x) > t)
		= \norm{x}^{k-d} \ind_{\ooint{0}{\widehat{\phi}^{-1}(-\log t)}}(\norm{x})
		= h_{1,t}(\norm{x})
		\label{Eq:contraction_X_up_density}
	\end{equation}
	for $x \neq 0$, where we define $h_{1,t}: \R_+ \ra \R_+$ by
	\begin{equation*}
		h_{1,t}(r)
		:= r^{k-d} \ind_{\ooint{0}{\widehat{\phi}^{-1}(-\log t)}}(r) .
	\end{equation*}
	Denote by $\lambda_d$ the Lebesgue measure on $(\R^d, \B(\R^d))$, by $\sph^{d-1} := \{x \in \R^d \;\colon \norm{x} = 1\}$ the $(d-1)$-sphere and by $\sigma_d$ the surface measure on $(\sph^{d-1}, \B(\sph^{d-1}))$, that is, the unique rotationally invariant measure whose total mass $\omega_d := \sigma_d(\sph^{d-1})$ equals the surface area of $\sph^{d-1}$.
	
	Then, by \eqref{Eq:contraction_X_up_density} and Corollary \ref{Cor:sampling_in_polar_coords}, the $X$-update of $k$-PSS, given $T_n = t$, can be implemented in polar coordinates as $X_n := R_n \, \Theta_n$ by sampling $(R_n, \Theta_n)$ from the joint distribution with density
	\begin{equation*}
		(r, \theta)
		\mapsto h_{2,t}(r)
		:= r^{d-1} h_{1,t}(r) \ind_{\R_+}(r)
		= r^{k-1} \ind_{\ooint{0}{\widehat{\phi}^{-1}(-\log t)}}(r)
	\end{equation*}
	w.r.t.~$\lambda_1 \otimes \sigma_d$. As this density is constant in $\theta$, it can in turn be implemented by independently sampling $\Theta_n$ from $\U(\sph^{d-1})$ and $R_n$ from the distribution with density $h_{2,t}$. To couple the sampling of $R_n$, we apply the inversion method (also called \textit{inverse transform sampling}). For that we first need to compute the c.d.f.~$F_t$ corresponding to density $h_{2,t}$. Let $s \in \ccint{0}{\widehat{\phi}^{-1}(-\log t)}$, then
	\begin{equation*}
		F_t(s)
		= \frac{\int_{-\infty}^s h_{2,t}(r) \d r}{\int_{-\infty}^{\infty} h_{2,t}(r) \d r}
		= \frac{\int_0^s r^{k-1} \d r}{\int_0^{\widehat{\phi}^{-1}(-\log t)} r^{k-1} \d r}
		= \frac{\left[\frac{1}{k} r^k\right]_0^s}{\left[\frac{1}{k} r^k\right]_0^{\widehat{\phi}^{-1}(-\log t)}}
		= \left( \frac{s}{\widehat{\phi}^{-1}(-\log t)} \right)^k_{\textbf{.}}
	\end{equation*}
	From this it is easy to see that $F_t$ restricted to $\ccint{0}{\widehat{\phi}^{-1}(-\log t)}$ maps bijectively onto $\ccint{0}{1}$ and that the inverse $F_t^{-1}: [0,1] \ra \ccint{0}{\widehat{\phi}^{-1}(-\log t)}$ of this restriction is given by
	\begin{equation*}
		F_t^{-1}(u)
		= \widehat{\phi}^{-1}(- \log t) u^{1/k} ,
		\qquad u \in \ccint{0}{1} .
	\end{equation*}
	
	Denote by $P$ the transition kernel of $k$-PSS for $\nu$ and fix $x,y \in \R^d$ with $\norm{x}, \norm{y} \in \ooint{0}{\kappa}$. Intuitively, we can now couple the sampling from $P(x,\cdot)$ and $P(y,\cdot)$ as follows:
	\begin{enumerate}
		\item Sample $U_1 \sim \U(\ooint{0}{1})$, call the result $u_1$ and set $t_x := u_1 \eta_{k,1}(x)$ and $t_y := u_1 \eta_{k,1}(y)$.
		\item Sample $U_2 \sim \U(\ooint{0}{1})$, call the result $u_2$ and set $r_x := F_{t_x}^{-1}(u_2)$ and $r_y := F_{t_y}^{-1}(u_2)$.
		\item Sample $\Theta \sim \U(\sph^{d-1})$, call the result $\theta$ and return $r_x \theta$ and $r_y \theta$ as the new samples.
	\end{enumerate}
	This corresponds to the coupling $\gamma_{x,y}$ determined by
	\begin{align*}
		\gamma_{x,y}(A \times B) 
		:= \omega_d^{-1} \int_0^1 \int_0^1 \int_{\sph^{d-1}} &\ind_A\!\left( F_{u_1 \eta_{k,1}(x)}^{-1}(u_2) \theta \right) \\
		\cdot\; &\ind_B\!\left( F_{u_1 \eta_{k,1}(y)}^{-1}(u_2) \theta \right) \sigma_d(\d \theta) \d u_2 \d u_1
	\end{align*}
	for $A, B \in \B(\R^d)$. Using this coupling, we can estimate the Wasserstein distance as
	\begin{align*}
		&\W(P(x,\cdot),P(y,\cdot)) \\
		&\leq \int_{\R^d \times \R^d} \norm{\tilde{x} - \tilde{y}} \gamma_{x,y}(\d \tilde{x} \times \d \tilde{y}) \\ 
		&= \omega_d^{-1} \int_0^1 \int_0^1 \int_{\sph^{d-1}} \norm{ F_{u_1 \eta_{k,1}(x)}^{-1}(u_2) \theta - F_{u_1 \eta_{k,1}(y)}^{-1}(u_2) \theta } \sigma_d(\d \theta) \d u_2 \d u_1 \\ 
		&= \int_0^1 \int_0^1 \abs{ F_{u_1 \eta_{k,1}(x)}^{-1}(u_2) - F_{u_1 \eta_{k,1}(y)}^{-1}(u_2) } \d u_2 \d u_1 \\ 
		&= \int_0^1 \int_0^1 \abs{ \widehat{\phi}^{-1}\big(-\log(u_1 \eta_{k,1}(x))\big) - \widehat{\phi}^{-1}\big(-\log(u_1 \eta_{k,1}(y))\big) } u_2^{1/k} \d u_2 \d u_1 \\ 
		&= \frac{k}{k+1} \int_0^1 \abs{ \widehat{\phi}^{-1}\big(\phi(\norm{x}) - \log u_1\big) - \widehat{\phi}^{-1}\big(\phi(\norm{y}) - \log u_1\big) } \d u_1 \\ 
		&\leq \frac{k}{k+1} \abs{\norm{x} - \norm{y}} , 
	\end{align*}
	where we use 
	\begin{equation}
		\int_0^1 u^{1/p} \d u
		= \left[ \frac{1}{1 + 1/p} u^{1+1/p} \right]_0^1 
		= \frac{1}{1 + 1/p}
		= \frac{p}{p+1} ,
		\qquad p \in \R \setminus \{-1\}
		\label{Eq:u_integral_formula}
	\end{equation}
	to compute the $u_2$-integral and apply Lemma \ref{Lem:non_expansion_prop} and monotonicity of the Lebesgue integral to obtain the inequality in the last line. Furthermore, by triangle inequality of the Euclidean norm, one has $\abs{\norm{x} - \norm{y}} \leq \norm{x-y}$, which is sharp whenever $x$ and $y$ lie on the same ray emanating from the origin. Thus we now have
	\begin{equation*}
		\Dob(P)
		= \sup_{x \neq y} \frac{\W(P(x,\cdot),P(y,\cdot))}{\norm{x - y}}
		\leq \sup_{x \neq y} \frac{\frac{k}{k+1} \abs{\norm{x} - \norm{y}}}{\norm{x - y}}
		= \frac{k}{k+1}_,
	\end{equation*}
	proving that $P$ is Wasserstein contractive with rate $\rho = k/(k+1)$.
\end{proof}

\begin{discussion} \label{Dis:contraction_result} \
	\begin{enumerate}[(a)]
		\item We can alternatively state Theorem \ref{Thm:contraction} as follows. If $\nu \in \Pc(\R^d)$ is a distribution with density $\eta$ and $k > 0$ a constant such that $\eta_{k,1}$ is rotationally invariant and log-concave (meaning that it takes the shape \eqref{Eq:D_k_likelihood_factor}), then $k$-PSS for $\nu$ is Wasserstein contractive with rate $\rho = k/(k+1)$. This interpretation would become particularly relevant if one were to view the function $\eta_{k,0}$ as an (improper) prior and ask about the performance of the corresponding slice sampler for rotationally invariant, log-concave likelihood functions $\eta_{k,1}$. An analogous interpretation also makes sense for two of our later results, Theorems \ref{Thm:k-PSS_rot_inv} and \ref{Thm:k-PSS_non_rot}.
		\item Letting $k = d$ in Theorem \ref{Thm:contraction}, we retrieve precisely \cite[Theorem 2.1]{SlavaUSS}, the result we are generalizing, which established Wasserstein contraction with rate $\rho = d/(d+1)$ of USS for rotationally invariant, log-concave densities.
		\item With $k = 1$, Theorem \ref{Thm:contraction} establishes Wasserstein contraction of PSS, for which, to our knowledge, no such result had been proven yet, with rate $\rho = 1/2$. Though the property is only shown for densities that have a pole of order $(d-1)$ at zero, we still deem it noteworthy that the contraction rate is $1/2$, regardless of the dimension $d$ of the underlying space $\R^d$ (in contrast to the result for USS). 
		\item Furthermore, we make note of the fact that the densities for which Wasserstein contraction of PSS is proven are not necessarily log-concave, not even along rays emanating from the origin. For example, with $\kappa = \infty$ and $\phi(r) := \alpha \, r$ for an $\alpha > 0$, we get $\log \eta(r \theta) = (1-d) \log r - \alpha \, r$ for $r \in \R_{>0}$, $\theta \in \sph^{d-1}$, which for $d > 1$ is actually strictly convex in $r$. This is particularly noteworthy because even the broadest previously known quantitative convergence result for PSS \cite[Theorem 7]{PolarSS} required log-concavity (at least along rays emanating from the origin). In this sense, our result broadens the range of target distributions in which quantitative convergence guarantees for PSS are available. We will further extend this range in the following subsection.
		\item Finally, but perhaps most importantly, as $k$-PSS continuously interpolates between USS and PSS through the parameter $k$, the contraction rate $\rho = k/(k+1)$ provided by Theorem \ref{Thm:contraction} interpolates accordingly, which shows the interpolation between the samplers to be meaningful in the context of convergence rates.
	\end{enumerate}
\end{discussion}

\subsection{Heavy-Tailed Targets} \label{Sec:contraction_heavy-tailed}

Some numerical experiments recently published in \cite{GPSS} suggest that PSS performs extremely well for heavy-tailed target distributions in high dimension (provided that they are centered on the origin), whereas USS performs very poorly for the same targets. However, as far as we can tell, no quantitative theoretical results are available to substantiate this hypothesis. Even beyond these specifics, the quantitative theoretical properties of slice sampling for heavy-tailed targets appear to be entirely unexplored.

We therefore deem it worthwhile to detail our investigation of the theoretical behavior of PSS and USS for certain -- admittedly somewhat inflexible -- classes of heavy-tailed targets. The results we state and prove in this section, Theorems \ref{Thm:USS_contraction_std_multiv_t} and \ref{Thm:contraction_heavy-tailed}, as well as Propositions \ref{Prop:USS_contraction_sharp} and \ref{Prop:contraction_heavy-tailed_sharp}, should serve to provide some first answers about the samplers' properties and may well inspire future work on the topic.

Arguably the textbook example of multivariate heavy-tailed distributions are the \textit{multivariate $t$-distributions}, which in their standard form have unnormalized densities of the form \eqref{Eq:std_multiv_t} on $\R^d$, where $m > 0$ is called the \textit{degrees of freedom} parameter. Considering only USS, it is feasible to analyze the behavior for \eqref{Eq:std_multiv_t}, and we obtain the following result.

\begin{theorem} \label{Thm:USS_contraction_std_multiv_t}
	Let $\nu \in \Pc(\R^d)$ be a distribution with density $\eta$ given by
	\begin{equation}
		\eta(x) 
		= \left(1 + \frac{1}{m} \norm{x}^2\right)^{-(d+m)/2} ,
		\label{Eq:std_multiv_t}
	\end{equation}
	where $m \in \ooint{1}{\infty}$. Then USS for $\nu$ is Wasserstein contractive with rate
	\begin{equation*}
		\rho = \frac{d (d+m)}{(d+1) (d+m-1)}_{\textbf{.}}
	\end{equation*}
\end{theorem}
\begin{proof}
	We begin by observing for $x \in \R^d$, $t > 0$, that
	\begin{align*}
		\eta(x) > t \quad
		&\LRA \quad 1 + \frac{1}{m} \norm{x}^2 < t^{-2/(d+m)} \\
		&\LRA \quad \norm{x} < \sqrt{ m \, t^{-2/(d+m)} - m} .
	\end{align*}
	Thus the level set at threshold $t$ is 
	\begin{equation*}
		\{x \in \R^d \mid \eta(x) > t\}
		= \Bf(\sqrt{ m \, t^{-2/(d+m)} - m}) ,
	\end{equation*}
	where $\Bf(\kappa)$ denotes the open zero-centered Euclidean ball of radius $\kappa$, for any $\kappa > 0$. Note that for USS the $X$-update consists of uniform sampling from the slice. Moreover, it is well-known (and also a straightforward consequence of Corollary \ref{Cor:sampling_in_polar_coords} and inversion method) that if $\kappa > 0$, $U \sim \U(\!\coint{0}{1})$ and $\Theta \sim \U(\sph^{d-1})$, then $U^{1/d} \kappa \Theta \sim \U(\Bf(\kappa))$. We make use of this in the construction of our coupling.
	
	Denote by $P$ the transition kernel of USS for $\nu$. For any $x, y \in \R^d$, we intuitively couple $P(x,\cdot)$ and $P(y,\cdot)$ as follows:
	\begin{enumerate}
		\item Sample $U_1 \sim \U(\ooint{0}{1})$, call the result $u_1$ and set 
		\begin{align*}
			&t_x := u_1 \eta(x) = u_1 \left(1 + \frac{1}{m} \norm{x}^2\right)^{-(d+m)/2} , \\
			&t_y := u_1 \eta(y) = u_1 \left(1 + \frac{1}{m} \norm{y}^2\right)^{-(d+m)/2} .
		\end{align*}
		\item Sample $U_2 \sim \U(\coint{0}{1})$, call the result $u_2$ and set 
		\begin{align*}
			r_x(u_1,u_2) 
			&:= u_2^{1/d} \sqrt{m t_x^{-2/(d+m)} - m}
			= u_2^{1/d} \sqrt{m u_1^{-2/(d+m)} \left( 1 + \tfrac{1}{m} \norm{x}^2 \right) - m} \\
			&= u_2^{1/d} \sqrt{\left(u_1^{-1/(d+m)} \norm{x} \right)^2 + m u_1^{-2/(d+m)} - m} , \\
			r_y(u_1,u_2) 
			&:= u_2^{1/d} \sqrt{\left(u_1^{-1/(d+m)} \norm{y} \right)^2 + m u_1^{-2/(d+m)} - m} .
		\end{align*}
		\item Sample $\Theta \sim \U(\sph^{d-1})$, call the result $\theta$ and set $\tilde{x} := r_x(u_1,u_2) \theta$ and $\tilde{y} := r_y(u_1,u_2) \theta$.
		\item Return $\tilde{x}$ and $\tilde{y}$ as the new states
	\end{enumerate}
	This corresponds to the coupling $\gamma_{x,y}$ determined by
	\begin{equation}
		\gamma_{x,y}(A \times B) 
		= \omega_d^{-1} \int_0^1 \int_0^1 \int_{\sph^{d-1}}
		\ind_A(r_x(u_1,u_2) \theta) \ind_B(r_y(u_1,u_2) \theta) \sigma_d(\d \theta) \d u_2 \d u_1
		\label{Eq:coupling_heavy-tailed}
	\end{equation}
	for $A, B \in \B(\R^d)$. Hence we obtain
	\begin{align*}
		\W(P(x,\cdot), P(y,\cdot))
		&\leq \int_{\R^d \times \R^d} \norm{\tilde{x} - \tilde{y}} \gamma_{x,y}(\d \tilde{x} \times \d \tilde{y}) \\ 
		&= \omega_d^{-1} \int_0^1 \int_0^1 \int_{\sph^{d-1}}
		\norm{r_x(u_1,u_2) \theta - r_y(u_1,u_2) \theta} \sigma_d(\d \theta) \d u_2 \d u_1 \\ 
		&= \int_0^1 \int_0^1 \abs{r_x(u_1,u_2) - r_y(u_1,u_2)} \d u_2 \d u_1 \\ 
		&= \int_0^1 u_2^{1/d} \d u_2 \cdot \int_0^1 \bigg\vert \sqrt{\left(u_1^{-1/(d+m)} \norm{x}\right)^2 + m u_1^{-2/(d+m)} - m} \\
		&\qquad\qquad\qquad\qquad\quad - \sqrt{\left(u_1^{-1/(d+m)} \norm{y}\right)^2 + m u_1^{-2/(d+m)} - m} \; \bigg\vert \d u_1 \\ 
		&\leq \int_0^1 u_2^{1/d} \d u_2 \cdot \int_0^1 \abs{ u_1^{-1/(d+m)} \norm{x} - u_1^{-1/(d+m)} \norm{y} } \d u_1 \\ 
		&= \frac{d}{d+1} \cdot \frac{d+m}{d+m-1} \cdot \abs{\norm{x} - \norm{y}} . 
	\end{align*}
	Note that, somewhat analogous to the proof of Theorem \ref{Thm:contraction}, the second inequality follows by Lemma \ref{Lem:non_expansion_prop}, here applied with $\phi(s) := s^2$, and monotonicity of the Lebesgue integral. The equality in the last line follows by \eqref{Eq:u_integral_formula}, applied with $p = d$ and $p = -(d+m)$. By the same arguments as in Theorem \ref{Thm:contraction}, the above estimation yields our claim.
\end{proof}

To demonstrate the sensitivity of our estimations in Theorem \ref{Thm:USS_contraction_std_multiv_t}, we complement them by the following sharpness result.

\begin{proposition} \label{Prop:USS_contraction_sharp}
	The contraction rate estimate provided by Theorem \ref{Thm:USS_contraction_std_multiv_t} is sharp. That is, for any $m > 1$ the transition kernel $P$ of USS for \eqref{Eq:std_multiv_t} satisfies
	\begin{equation}
		\Dob(P)
		= \frac{d(d+m)}{(d+1)(d+m-1)}_{\textbf{.}}
		\label{Eq:Dob_USS_std_multiv_t}
	\end{equation}
\end{proposition}
\begin{proof}
	Theorem \ref{Thm:USS_contraction_std_multiv_t} shows \eqref{Eq:Dob_USS_std_multiv_t} with ``$\leq$'' in place of equality. Thus it only remains to show ``$\geq$''. For this we make use of the Kantorovich-Rubinstein theorem \cite[Section 1.2]{Villani}, which states that
	\begin{equation*}
		\W(P(x,\cdot), P(y,\cdot))
		= \sup_{\norm{g}_{\text{Lip}} \leq 1} \abs{ \int_{\R^d} g(z) P(x,\d z) - \int_{\R^d} g(z) P(y,\d z) } ,
	\end{equation*}
	where
	\begin{equation*}
		\norm{g}_{\text{Lip}} := \sup_{z_1 \neq z_2 \in \R^d} \frac{\abs{g(z_1) - g(z_2)}}{\norm{z_1 - z_2}}_{\textbf{.}}
	\end{equation*}
	Note that, by triangle inequality of $\norm{\cdot}$, for $g := \norm{\cdot}$ one has $\norm{g}_{\text{Lip}} = 1$.
	
	The analysis in the proof of Theorem \ref{Thm:USS_contraction_std_multiv_t} shows
	\begin{equation*}
		P(x,A)
		= \omega_d^{-1} \int_0^1 \int_0^1 \int_{\sph^{d-1}} \ind_A(r_x(u_1,u_2) \theta) \sigma_d(\d \theta) \d u_2 \d u_1 ,
	\end{equation*}
	so that
	\begin{align*}
		\int_{\R^d} \norm{z} P(x, \d z) 
		&= \omega_d^{-1} \int_0^1 \int_0^1 \int_{\sph^{d-1}} \norm{r_x(u_1,u_2) \theta} \sigma_d(\d \theta) \d u_2 \d u_1 \\
		&= \int_0^1 \int_0^1 r_x(u_1,u_2) \d u_2 \d u_1 \\
		&= \frac{d}{d+1} \int_0^1 \sqrt{\left( u_1^{-1/(d+m)} \norm{x}\right)^2 + m u_1^{-2/(d+m)} - m} \; \d u_1 .
	\end{align*}
	We now consider the points $x = r \theta_0$, $y = \tfrac{1}{2} r \theta_0$ for varying $r \in \R_+$ and arbitrary but fixed $\theta_0 \in \sph^{d-1}$. By the Kantorovich-Rubinstein theorem, we obtain
	\begin{align*}
		\Dob(P) 
		&\geq \sup_{x \neq y} \frac{\abs{\int_{\R^d} \norm{z} P(x,\d z) - \int_{\R^d} \norm{z} P(y,\d z)}}{\norm{x - y}} \\ 
		&\geq \lim_{r \ra \infty} \frac{\abs{\int_{\R^d} \norm{z} P(r \theta_0,\d z) - \int_{\R^d} \norm{z} P(\tfrac{1}{2} r \theta_0,\d z)}}{\norm{r \theta_0 - \tfrac{1}{2} r \theta_0}} \\ 
		&= \frac{2 d}{d+1} \lim_{r \ra \infty} \bigg\vert \int_0^1 \sqrt{\left( u_1^{-1/(d+m)} \right)^2 + \frac{m u_1^{-2/(d+m)} - m}{r^2}} \d u_1 \\
		&\qquad\qquad\qquad\qquad - \int_0^1 \sqrt{\left( \frac{1}{2} u_1^{-1/(d+m)} \right)^2 + \frac{m u_1^{-2/(d+m)} - m}{r^2}} \; \d u_1 \bigg\vert \\
		&= \frac{2 d}{d+1} \abs{ \int_0^1 \left( u_1^{-1/(d+m)} - \frac{1}{2} u_1^{-1/(d+m)} \right) \d u_1 } 
		= \frac{d}{d+1} \int_0^1 u_1^{-1/(d+m)} \d u_1 \\
		&= \frac{d}{d+1} \cdot \frac{d+m}{d+m-1} ,
	\end{align*}
	which is precisely what we set out to prove. Note that the limes can be pulled into the integrals by the monotone convergence theorem.
\end{proof}

Unfortunately, for $k$-PSS besides USS, so in particular for PSS, the density factors $\eta_{k,1}$ for $\eta$ as in \eqref{Eq:std_multiv_t} have intractable level sets, which prevents us from extending the analysis of Theorem \ref{Thm:USS_contraction_std_multiv_t} to these cases. However, we may instead consider a class of distributions that have a simpler structure, while possessing tails asymptotically equivalent to those of the standard multivariate $t$-distributions \eqref{Eq:std_multiv_t}. Namely, we consider the target densities \eqref{Eq:heavy-tailed_simple}, which also have a parameter $m > 0$ that serves the same role as in \eqref{Eq:std_multiv_t}. We obtain the following result.

\begin{theorem} \label{Thm:contraction_heavy-tailed}
	For an $\epsi > 0$ let $\nu \in \Pc(\R^d \setminus \Bf(\epsi)) \subset \Pc(\R^d)$ be a distribution with density $\eta$ given by
	\begin{equation}
		\eta(x)
		= \norm{x}^{-(d+m)} \I(\norm{x} \geq \epsi)
		\label{Eq:heavy-tailed_simple}
	\end{equation}
	for some $m > 1$. Then $k$-PSS for $\nu$ is Wasserstein contractive with rate
	\begin{equation*}
		\rho = \frac{k (k+m)}{(k+1) (k+m-1)}
	\end{equation*}
	for any $k \geq 1$.
\end{theorem}
\begin{proof}
	Arbitrarily fix $k \geq 1$. Note that for any $r \geq \epsi$ and $\theta \in \sph^{d-1}$, the distribution-dependent factor of the $k$-PSS factorization becomes $\eta_{k,1}(r \theta) = r^{-(k+m)}$. Thus, for any $\theta \in \sph^{d-1}$ and $t \in \ooint{0}{\epsi^{-(k+m)}}$, we get
	\begin{equation*}
		\{r \in \R_+ \mid \eta_{k,1}(r \theta) > t \}
		= \coint{\epsi}{t^{-1/(k+m)}} .
	\end{equation*}
	As in the proof of Theorem \ref{Thm:contraction}, we decompose $k$-PSS into polar coordinates and couple the sampling of the radii via inversion method. To this end, we note that by Corollary \ref{Cor:sampling_in_polar_coords}, given threshold $t$, radius and direction  need to be drawn from the joint distribution with density
	\begin{equation*}
		(r, \theta) 
		\mapsto r^{k-1} \I(\eta_{k,1}(r \theta) > t) \ind_{\R_+}(r)
		= r^{k-1} \ind_{\coint{\epsi}{t^{-1/(k+m)}}}(r) .
	\end{equation*}
	This corresponds to the c.d.f.~$F_t$ that for $s \in \ccint{\epsi}{t^{-1/(k+m)}}$ is given by
	\begin{equation*}
		F_t(s)
		:= \frac{\int_{\epsi}^s r^{k-1} \d r}{\int_{\epsi}^{t^{-1/(k+m)}} r^{k-1} \d r}
		= \frac{\left[ \frac{1}{k} r^k \right]_{\epsi}^s}{\left[ \frac{1}{k} r^k \right]_{\epsi}^{t^{-1/(k+m)}}}
		= \frac{s^k - \epsi^k}{t^{-k/(k+m)} - \epsi^k}_{\textbf{.}}
	\end{equation*}
	Observe that $F_t$ bijectively maps $\ccint{\epsi}{t^{-1/(k+m)}}$ onto $\ccint{0}{1}$ and that the inverse $F_t^{-1}$ of this restriction is given by
	\begin{equation*}
		F_t^{-1}(u)
		= ( (t^{-k/(k+m)} - \epsi^k) u + \epsi^k)^{1/k}
		= ( u \, t^{-k/(k+m)} + (1 - u) \epsi^k )^{1/k} ,
		\quad u \in \ccint{0}{1} .
	\end{equation*}
	By introducing the trivially invertible auxiliary function $h: \R_+ \ra \R_+, \; r \mapsto r^k$, we may alternatively write
	\begin{equation*}
		F_t^{-1}(u)
		= h^{-1}(h(u^{1/k} t^{-1/(k+m)}) + (1 - u) \epsi^k) .
	\end{equation*}
	
	Denote by $P$ the transition kernel of $k$-PSS for $\nu$. For any $x, y \in \R^d \setminus \Bf(\epsi)$, we intuitively construct our coupling of $P(x,\cdot)$ and $P(y,\cdot)$ as follows:
	\begin{enumerate}
		\item Sample $U_1 \sim \U(\ooint{0}{1})$, call the result $u_1$ and set 
		\begin{equation*}
			t_x := u_1 \eta_{k,1}(x) = u_1 \norm{x}^{-(k+m)}
			\quad \text{and} \quad 
			t_y := u_1 \eta_{k,1}(y) = u_1 \norm{y}^{-(k+m)} .
		\end{equation*}
		\item Sample $U_2 \sim \U(\coint{0}{1})$, call the result $u_2$ and set 
		\begin{align*}
			&r_x(u_1,u_2) := F_{t_x}^{-1}(u_2) 
			= h^{-1}(h(u_2^{1/k} u_1^{-1/(k+m)} \norm{x}) + (1 - u_2) \epsi^k) , \\
			&r_y(u_1,u_2) := F_{t_y}^{-1}(u_2) 
			= h^{-1}(h(u_2^{1/k} u_1^{-1/(k+m)} \norm{y}) + (1 - u_2) \epsi^k) .
		\end{align*}
		\item Sample $\Theta \sim \U(\sph^{d-1})$, call the result $\theta$ and set $\tilde{x} := r_x(u_1,u_2) \theta$ and $\tilde{y} := r_y(u_1,u_2) \theta$.
		\item Return $\tilde{x}$ and $\tilde{y}$ as the new states
	\end{enumerate}
	
	The resulting coupling again takes the shape \eqref{Eq:coupling_heavy-tailed}. Thus, omitting some steps that coincide with ones in Theorem \ref{Thm:USS_contraction_std_multiv_t}, we get
	\begin{align*}
		\W(P(x,\cdot), P(y,\cdot))
		&\leq \int_0^1 \int_0^1 \abs{r_x(u_1,u_2) - r_y(u_1,u_2)} \d u_2 \d u_1 \\ 
		&= \int_0^1 \int_0^1 \vert h^{-1}(h(u_2^{1/k} u_1^{-1/(k+m)} \norm{x}) + (1 - u_2) \epsi^k) \\
		&\qquad\qquad\quad - h^{-1}(h(u_2^{1/k} u_1^{-1/(k+m)} \norm{y}) + (1 - u_2) \epsi^k) \vert \d u_2 \d u_1 \\ 
		&\leq \int_0^1 \int_0^1 \abs{ u_2^{1/k} u_1^{-1/(k+m)} \norm{x} - u_2^{1/k} u_1^{-1/(k+m)} \norm{x} } \d u_2 \d u_1 \\ 
		&= \int_0^1 u_2^{1/k} \d u_2 \cdot \int_0^1 u_1^{-1/(k+m)} \d u_1 \cdot \abs{\norm{x} - \norm{y}} \\ 
		&= \frac{k}{k+1} \cdot \frac{k+m}{k+m-1} \cdot \abs{\norm{x} - \norm{y}} . 
	\end{align*}
	Note that, as in Theorem \ref{Thm:USS_contraction_std_multiv_t}, the second inequality follows by Lemma \ref{Lem:non_expansion_prop}, here applied with $\phi := h$, which is convex thanks to our assumption that $k \geq 1$. The final equality again follows by \eqref{Eq:u_integral_formula}, here applied with $p = k$ and $p = -(k+m)$. By the same arguments as in the proof of Theorem \ref{Thm:contraction}, the above estimation yields our claim.
\end{proof}

Analogous to what we did for Theorem \ref{Thm:USS_contraction_std_multiv_t} with Proposition \ref{Prop:USS_contraction_sharp}, we can complement the estimates of Theorem \ref{Thm:contraction_heavy-tailed} by a sharpness guarantee.

\begin{proposition} \label{Prop:contraction_heavy-tailed_sharp}
	The contraction rate estimate provided by Theorem \ref{Thm:contraction_heavy-tailed} is sharp. That is, for any $m > 1$ and $k \geq 1$, the transition kernel $P$ of $k$-PSS for \eqref{Eq:heavy-tailed_simple} satisfies
	\begin{equation*}
		\Dob(P)
		= \frac{k (k+m)}{(k+1) (k+m-1)} .
	\end{equation*}
\end{proposition}
\begin{proof}
	This follows by the same proof strategy as Proposition \ref{Prop:USS_contraction_sharp}, though of course the differing transition kernel causes some of the details to change. For the reader's convenience we provide the part of the estimation affected by this (using notation and auxiliary results from the proof of Theorem \ref{Thm:contraction_heavy-tailed}).
	
	\begin{align*}
		\Dob(P)
		&\geq \lim_{r \ra \infty} \frac{\abs{\int_{\R^d} \norm{z} P(r \theta_0,\d z) - \int_{\R^d} \norm{z} P(\tfrac{1}{2} r \theta_0,\d z)}}{\norm{r \theta_0 - \tfrac{1}{2} r \theta_0}} \\ 
		&= 2 \lim_{r \ra \infty} \frac{1}{r} \bigg\vert \int_0^1 \int_0^1 \big( h^{-1}(h(u_2^{1/k} u_1^{-1/(k+m)} r) + (1 - u_2) \epsi^k) \\
		&\qquad\qquad\qquad\qquad\qquad - h^{-1}(h(\tfrac{1}{2} u_2^{1/k} u_1^{-1/(k+m)} r) + (1 - u_2) \epsi^k) \big) \d u_2 \d u_1 \bigg\vert \\
		&= 2 \lim_{r \ra \infty} \bigg\vert \int_0^1 \int_0^1 \bigg( h^{-1}\left(h(u_2^{1/k} u_1^{-1/(k+m)}) + \frac{(1 - u_2) \epsi^k}{r^k}\right) \\
		&\qquad\qquad\qquad\qquad\quad\; - h^{-1}\left(h(\tfrac{1}{2} u_2^{1/k} u_1^{-1/(k+m)}) + \frac{(1 - u_2) \epsi^k}{r^k}\right) \bigg) \d u_2 \d u_1 \bigg\vert \\
		&= 2 \abs{ \int_0^1 \int_0^1 \big( u_2^{1/k} u_1^{-1/(k+m)} - \tfrac{1}{2} u_2^{1/k} u_1^{-1/(k+m)} \big) \d u_2 \d u_1 } \\
		&= \int_0^1 u_2^{1/k} \d u_2 \cdot \int_0^1 u_1^{-1/(k+m)} \d u_1
		= \frac{k}{k+1} \cdot \frac{k+m}{k+m-1} 
		\qedhere
	\end{align*}
\end{proof}

\begin{discussion} \
	\begin{enumerate}[(a)]
		\item We note that, much like in Theorem \ref{Thm:contraction}, the contraction rate shown in Theorem \ref{Thm:contraction_heavy-tailed} only depends on the sample space dimension $d$ if $k$ depends on it. Consequently the rate obtained for PSS (i.e.~$k=1$), which is
		\begin{equation*}
			\rho = \frac{m+1}{2m}_{\textbf{,}}
		\end{equation*}
		is dimension-independent. We note that it tends to $1/2$ as $m \ra \infty$.
		\item For USS (i.e.~$k=d$) on the other hand, the rate obtained in both Theorems \ref{Thm:USS_contraction_std_multiv_t} and \ref{Thm:contraction_heavy-tailed} is the dimension-dependent
		\begin{equation}
			\rho = \frac{d (d+m)}{(d+1) (d+m-1)}_{\textbf{.}}
			\label{Eq:USS_contraction_rate_ht}
		\end{equation}
		\item That the contraction rates of USS for \eqref{Eq:std_multiv_t} and \eqref{Eq:heavy-tailed_simple} coincide suggests that the relevant properties of the former are retained when approximating it by the latter -- especially because, by the sharpness results of Propositions \ref{Prop:USS_contraction_sharp} and \ref{Prop:contraction_heavy-tailed_sharp}, we do not merely have coinciding \textit{rate estimates}, but actually coinciding \textit{rates}. 
		We also note here that the weight the density \eqref{Eq:heavy-tailed_simple} places on its tails can be freely controlled by the parameter $\epsi > 0$ (which does not affect the contraction rates), so that in particular this tail weight can be matched precisely with that of the corresponding density \eqref{Eq:std_multiv_t}. This further justifies using \eqref{Eq:heavy-tailed_simple} as a proxy for \eqref{Eq:std_multiv_t}.
		\item We note that \eqref{Eq:USS_contraction_rate_ht} tends to $d/(d+1)$ for any fixed $d$ as $m \ra \infty$. This is perhaps not too surprising, since in the same limit the density \eqref{Eq:std_multiv_t} approaches that of the multivariate standard normal distribution, which in turn fits (with $\kappa = \infty$ and $\phi(r) = r^2/2$) into Theorem \ref{Thm:contraction}, restricted to USS, where we also obtain Wasserstein contraction with rate $\rho = d/(d+1)$.
		\item On the other hand, for any fixed $m$, as $d \ra \infty$, the rate \eqref{Eq:USS_contraction_rate_ht} tends to $1$. This is particularly noteworthy to us because we are interested in the behavior of USS in high dimension. We emphasize that \eqref{Eq:USS_contraction_rate_ht} deteriorates fairly quickly with increasing dimension. For example, for target density \eqref{Eq:heavy-tailed_simple} with $m = 2$ in dimension $d = 100$, the contraction rate of PSS is $\rho = 3/4$, whereas that of USS is $\rho \approx 0.9999$.
		\item Our result has the caveat that it does not explain the empirically excellent performance (cf.~\cite[Section 6.1]{GPSS}) of PSS even for the standard multivariate Cauchy distribution, which results from \eqref{Eq:std_multiv_t} with $m = 1$, and thus corresponds in the simplified setting to \eqref{Eq:heavy-tailed_simple} with $m = 1$. We emphasize that the issue here is not an insufficiently detailed analysis: When plugging $m=1$ into the proofs of both Theorem \ref{Thm:contraction_heavy-tailed} and Proposition \ref{Prop:contraction_heavy-tailed_sharp}, they yield $\Dob(P) = 1$, meaning that PSS simply is not Wasserstein contractive for this target.
		
		This discrepancy between empirical and theoretical results gives reason to believe that the rates we obtain w.r.t.~the Wasserstein distance may not reflect those that hold (albeit so far unproven) w.r.t.~the total variation distance (i.e.~corresponding to spectral gaps).
	\end{enumerate}
\end{discussion}

\section{Spectral Gap Estimates} \label{Sec:gap_estimates}

\subsection{General Tool}

In \cite{SlavaUSS}, the authors developed a method to prove quantitative spectral gap estimates for USS. Roughly speaking, their method yields a spectral gap estimate for USS applied to a target density $\varrho$ whenever one is able to verify that the level set function $\ell: \R_{>0} \ra \R_+$ given by
\begin{equation*}
	\ell(t)
	:= \lambda_d(\{x \in \R^d \mid \varrho(x) > t\}) ,
	\qquad t > 0
\end{equation*}
has certain properties. In \cite{PSS_paper} this method was adapted to general slice samplers, using arbitrary factorization $\varrho = \varrho_0 \, \varrho_1$, which necessitated replacing $\ell$ by the \textit{generalized level set function} $\ell_{\varrho_0, \varrho_1}: \R_{>0} \ra \R_+$ given by
\begin{equation}
	\ell_{\varrho_0, \varrho_1}(t)
	:= \int_{\R^d} \varrho_0(x) \I(\varrho_1(x) > t) \d x ,
	\qquad t > 0 .
	\label{Eq:def_ell}
\end{equation}
Using the first of our Wasserstein contraction results, Theorem \ref{Thm:contraction}, we further generalize the method to allow for more fine-grained and in some cases even further-reaching results. For this we require the following auxiliary result from \cite{PSS_paper}.

\begin{lemma} \label{Lem:gap_identity_ell}
	Let $d_1, d_2 \in \N$ and let $\pi \in \Pc(\R^{d_1})$ and $\nu \in \Pc(\R^{d_2})$ be distributions with densities $\varrho$ and $\eta$ arbitrarily factorized as $\varrho = \varrho_0 \, \varrho_1$ and $\eta = \eta_0 \, \eta_1$, like in \eqref{Eq:ss_factorization}. Provided that $\ell_{\varrho_0,\varrho_1} \equiv \ell_{\eta_0,\eta_1}$, i.e.~if $\ell_{\varrho_0,\varrho_1}(t) = \ell_{\eta_0,\eta_1}(t)$ for all $t \in \R_{>0}$, one has
	\begin{equation*}
		\gap_{\pi}(P_{\pi})
		= \gap_{\nu}(P_{\nu}) ,
	\end{equation*}
	where $P_{\pi}$ and $P_{\nu}$ are the transition kernels of slice sampling for $\pi$ and $\nu$, respectively.
\end{lemma}
\begin{proof}
	See \cite[Theorem 3.5]{PSS_paper}.
\end{proof}

Following both \cite{SlavaUSS} and \cite{PSS_paper}, we conveniently package the conditions one needs to verify to obtain a spectral gap estimate by defining classes $\Lambda_k$ of all functions $\ell$ that satisfy them.

\begin{definition} \label{Def:Lambda_k}
	For any fixed $k \in \R_{>0}$, we define $\Lambda_k$ as the class of continuous functions $\ell: \R_{>0} \ra \R_+$ for which all of the following hold:
	\begin{enumerate}[(i)]
		\item $\lim_{t \ra \infty} \ell(t) = 0$ and $\sup \ell := \lim_{t \searrow 0} \ell(t) \in \ocint{0}{\infty}$ \label{Con:Lambda_k_i}
		\item $\ell$ restricted to $\supp(\ell) := \oointv{0}{\sup\{t \in \R_{>0} \mid \ell(t) > 0\}}$ is strictly decreasing \label{Con:Lambda_k_ii}
		\item the function $\ooint{0}{(\sup \ell)^{1/k}} \ra \R_{>0}, \; r \mapsto \ell^{-1}(r^k)$ is log-concave. \label{Con:Lambda_k_iii}
	\end{enumerate}
	Note that for functions $\ell_{\varrho_0,\varrho_1}$ as in \eqref{Eq:def_ell}, it is easy to see that condition (i) is always satisfied, so we ignore it when trying to show statements of the form $\ell_{\varrho_0,\varrho_1} \in \Lambda_k$.
\end{definition}

Following \cite{PSS_paper}, we note that the assumed continuity and conditions \eqref{Con:Lambda_k_i} and \eqref{Con:Lambda_k_ii} altogether imply that $\ell$ restricted to $\supp(\ell)$ maps bijectively onto $\ooint{0}{\sup \ell}$, which in turn implies the existence of the inverse function $\ell^{-1}: \ooint{0}{\sup \ell} \ra \supp(\ell)$ used in condition \eqref{Con:Lambda_k_iii}. It is easy to see that $\ell^{-1}$ is also strictly decreasing.

We are now prepared to state and prove our generalization of the method for establishing quantitative spectral gap estimates.

\begin{theorem} \label{Thm:gap_estimate}
	Let $\pi \in \Pc(\R^d)$ be a distribution with density $\varrho$ factorized as $\varrho = \varrho_0 \, \varrho_1$ like in \eqref{Eq:ss_factorization}. Suppose that $\ell_{\varrho_0,\varrho_1} \in \Lambda_k$ for some $k \in \R_{>0}$, then
	\begin{equation*}
		\gap_{\pi}(P_{\pi})
		\geq \frac{1}{k+1}_,
	\end{equation*}
	where $P_{\pi}$ is the transition kernel of slice sampling for $\pi$.
\end{theorem}
\begin{proof}
	Our proof strategy is to construct a distribution $\nu \in \D_k(\R)$ with density $\eta$, factorized according to $k$-PSS, such that $\ell_{\eta_{k,0},\eta_{k,1}} \equiv \ell_{\varrho_0,\varrho_1}$. Afterwards we just need to apply a number of earlier results to obtain the claimed estimate. We note that our proof strategy would also work if we instead constructed a $\nu \in \D_k(\R^m)$ for some $m > 1$ and that we simply chose $m = 1$ for simplicity.

	Fix any $k \in \R_{>0}$ such that $\ell_{\varrho_0,\varrho_1} \in \Lambda_k$, set
	\begin{equation*}
		\kappa
		:= \left( \frac{k}{2} \sup \ell_{\varrho_0,\varrho_1} \right)^{1/k}
		\in \, \ocint{0}{\infty}
	\end{equation*}
	and define $\phi: \ooint{0}{\kappa} \ra \R$ by
	\begin{equation*}
		\phi(r)
		:= -\log\left( \ell_{\varrho_0,\varrho_1}^{-1}\left( \frac{2}{k} r^k \right) \right)_{\textbf{.}}
	\end{equation*}
	Note that since $\phi$ is a composition of the strictly decreasing functions $-\log$ and $\ell_{\varrho_0,\varrho_1}^{-1}$ and the strictly increasing function $r \mapsto \frac{2}{k} r^k$, it is strictly increasing. Furthermore, since $\phi$ can also be viewed as a composition of the (by condition \eqref{Con:Lambda_k_iii} of Definition \ref{Def:Lambda_k}) convex function $r \mapsto -\log \ell_{\varrho_0,\varrho_1}^{-1}(r^k)$ and the linear function $r \mapsto (2/k)^{1/k} \, r$, it is convex as well. Hence it satisfies both assumptions of Definition \ref{Def:D_k_def} and therefore defines a distribution $\nu \in \D_k(\R)$ with density $\eta$ as in \eqref{Eq:D_k_def}.

	It is easy to verify that the inverse\footnote{See Lemma \ref{Lem:phi_inverse} for nomenclature.} $\phi^{-1}: I_{\phi} \ra D_{\phi}$ of $\phi$ is given by
	\begin{equation*}
		\phi^{-1}(s)
		= \left( \frac{k}{2} \ell_{\varrho_0,\varrho_1}(\exp(-s)) \right)^{1/k}_{\textbf{.}}
	\end{equation*}
	Furthermore, one has 
	\begin{equation*}
		\sup \phi
		= \lim_{r \nearrow \kappa} \phi(r)
		= \lim_{r^{\prime} \nearrow \sup \ell_{\varrho_0,\varrho_1}} -\log \ell_{\varrho_0,\varrho_1}^{-1}(r^{\prime}) 
		= \lim_{r^{\prime\prime} \searrow 0} -\log r^{\prime\prime} 
		= \infty ,
	\end{equation*}
	which yields $\widehat{\phi}^{-1} \equiv \phi^{-1}$ (see Lemma \ref{Lem:non_expansion_prop}). Plugging this into \eqref{Eq:L(t)_for_contraction} and the result into the definition of $\ell_{\eta_{k,0},\eta_{k,1}}$, we get
	\begin{align*}
		\ell_{\eta_{k,0},\eta_{k,1}}(t)
		&= \int_{\R} \abs{x}^{k-1} \ind_{\ooint{0}{\phi^{-1}(-\log t)}}(\abs{x}) \d x 
		= 2 \int_0^{\phi^{-1}(-\log t)} r^{k-1} \d r \\ 
		&= 2 \left[ \frac{1}{k} r^k \right]_0^{\phi^{-1}(-\log t)} 
		= \frac{2}{k} \phi^{-1}(-\log t)^k 
		= \ell_{\varrho_0,\varrho_1}(t) 
	\end{align*}
	for any $t \in \supp(\ell_{\eta_{k,0},\eta_{k,1}})$. Note that
	\begin{align*}
		&\sup\{t \in \R_{>0} \mid \ell_{\eta_{k,0},\eta_{k,1}}(t) > 0\}
		= \exp(-\inf \phi) 
		= \lim_{r \searrow 0} \exp(-\phi(r)) \\ 
		&= \lim_{r \searrow 0} \ell_{\varrho_0,\varrho_1}^{-1}\left( \frac{2}{k} r^k\right) 
		= \lim_{r^{\prime} \searrow 0} \ell_{\varrho_0,\varrho_1}^{-1}(r^{\prime}) 
		= \sup\{t \in \R_{>0} \mid \ell_{\varrho_0,\varrho_1}(t) > 0\} 
	\end{align*}
	and therefore $\supp(\ell_{\eta_{k,0},\eta_{k,1}}) = \supp(\ell_{\varrho_0,\varrho_1})$. Of course we also have
	\begin{equation*}
		\ell_{\eta_{k,0},\eta_{k,1}}(t) = 0 = \ell_{\varrho_0,\varrho_1}(t)
	\end{equation*}
	for thresholds $t$ too large to be contained in the supports. Overall, we have now shown $\ell_{\eta_{k,0},\eta_{k,1}} \equiv \ell_{\varrho_0,\varrho_1}$, so that Lemma \ref{Lem:gap_identity_ell} applies.

	Denote by $P_{\nu}$ the transition kernel of $k$-PSS for $\nu$. From Theorem \ref{Thm:contraction} we know that $P_{\nu}$ is Wasserstein contractive with rate $\rho = k/(k+1)$. Moreover, it is easy to see that all distributions in $\D_k(\R^d)$, for any $k > 0$ and $d \in \N$, have finite second moment, so in particular this must be true of $\nu \in \D_k(\R)$. Hence \eqref{Eq:contraction_to_gap} applies and we can conclude the proof by
	\begin{align*}
		\gap_{\pi}(P_{\pi})
		&\stackrel{\ref{Lem:gap_identity_ell}}{=} \gap_{\nu}(P_{\nu})
		\stackrel{\eqref{Eq:contraction_to_gap}}{\geq} \frac{1}{k+1}_{\textbf{.}}
	\end{align*}
\end{proof}

\begin{discussion} \label{Dis:gap_est_thm} \
	\begin{enumerate}[(a)]
		\item We emphasize that Definition \ref{Def:Lambda_k} is virtually identical to \cite[Definition 3.7]{PSS_paper}, which underlies the result we generalize here. The only difference is that the parameter $k$, which was restricted to $\N$ in \cite{PSS_paper}, is now allowed to take any value in $\R_{>0}$. Hence we generalize the method from a discrete sequence of classes $(\Lambda_k)_{k \in \N}$ to a continuous spectrum $(\Lambda_k)_{k \in \R_{>0}}$ that contains the sequence as a subset.
		\item This generalization is made possible by a subtle change in the proof strategy: Where the corresponding proof in \cite{PSS_paper} constructed for each $k \in \N$ a $\nu$ that was defined on $\R^k$ and then used the Wasserstein contraction result of \cite{SlavaUSS} for USS in that dimension, we instead stay in dimension one and use our own contraction result for $k$-PSS. Hence the increased flexibility of Theorem \ref{Thm:gap_estimate} (w.r.t.~the available choices of $k$) over the theorem's counterpart in \cite{PSS_paper} is essentially due to the fact that $k$-PSS is well-defined and proven to be a Wasserstein contraction for any $k \in \R_{>0}$, not just for integer values $k \in \N$.
		\item We will see in the following that our generalization is useful in the sense that we can find for any $k \in \R_{>0}$ a target distribution $\pi$ with density $\varrho$, factorized as in \eqref{Eq:ss_factorization}, satisfying $\ell_{\varrho_0,\varrho_1} \in \Lambda_k$ and $\ell_{\varrho_0,\varrho_1} \not\in \Lambda_{k^{\prime}}$ for all $k^{\prime} < k$.
		
		\item A particularly valuable aspect of the continuous spectrum in our result is that we allow the cases $k \in \ooint{0}{1}$, for which the theorem yields lower bounds on the spectral gap that lie in $\ooint{1/2}{1}$, whereas the result from \cite{SlavaUSS} could at best give a lower bound of $1/2$ (with $k=1$). This is of importance to both of the applications of Theorem \ref{Thm:gap_estimate} that we develop in the subsequent subsections.
	\end{enumerate}
\end{discussion}

In order to ease application of Theorem \ref{Thm:gap_estimate}, we provide in Appendix \ref{App:aux_res} two more results from \cite{PSS_paper} (Lemma \ref{Lem:convex_inverse} and Proposition \ref{Prop:alternative_cond}). In the remainder of this section, we present two non-trivial applications of Theorem \ref{Thm:gap_estimate} to $k$-PSS and a result for USS that follows by an analogous approach.

\subsection{Rotationally Invariant Targets}

Our first application of Theorem \ref{Thm:gap_estimate} concerns $k$-PSS for certain types of rotationally invariant target densities. For target densities $\varrho$ we obtain spectral gap estimates that depend on the degree to which their factor $\varrho_{k,1}$ is log-concave.

\begin{theorem} \label{Thm:k-PSS_rot_inv}
	Let $k \in \R_{>0}$ and let $\pi \in \Pc(\R^d)$ be a distribution with density $\varrho$ given by
	\begin{equation}
		\varrho(r \theta)
		= r^{k-d} \exp(-\phi(r^m)) \ind_{\ooint{0}{\kappa}}(r^m) ,
		\qquad r \in \R_+, \theta \in \sph^{d-1} ,
		\label{Eq:k-PSS_rot_inv}
	\end{equation}
	where $m \in \R_{>0}$, $\kappa \in \ocint{0}{\infty}$ and $\phi: \ooint{0}{\kappa} \ra \R$ is strictly increasing and convex and satisfies $\lim_{r \nearrow \kappa} \phi(r) = \infty$. Then one has
	\begin{equation*}
		\gap_{\pi}(P) \geq \frac{m}{k+m} ,
	\end{equation*}
	where $P$ is the transition kernel of $k$-PSS for $\pi$.
\end{theorem}
\begin{proof}
	Observe that for $t \in \ooint{0}{\exp(-\inf \phi)}$, $r > 0$ and $\theta \in \sph^{d-1}$ one has
	\begin{align*}
		\varrho_{k,1}(r \theta) > t \quad
		&\LRA \quad \exp(-\phi(r^m)) > t \quad \text{and} \quad r^m < \kappa \\
		&\LRA \quad r^m < \phi^{-1}(-\log t) \quad \text{and} \quad r^m < \kappa \\
		&\LRA \quad r^m < \phi^{-1}(-\log t) \\
		&\LRA \quad r < \phi^{-1}(-\log t)^{1/m} ,
	\end{align*}
	where the second and fourth equivalence rely on $\phi^{-1}$ and $r \mapsto r^{1/m}$ being strictly increasing, and the third equivalence on $\phi^{-1}$ mapping to $\ooint{0}{\kappa}$, s.t.~in particular $\phi^{-1} < \kappa$. Thus, by the polar coordinates formula (Proposition \ref{Prop:polar_coords}), for $t \in \ooint{0}{\exp(-\inf \phi)}$ we get
	\begin{align*}
		\ell_{\varrho_{k,0},\varrho_{k,1}}(t)
		&= \int_{\R^d} \norm{x}^{k-d} \I(\varrho_{k,1}(x) > t) \d x
		= \omega_d \int_0^{\infty} r^{k-1} \I(r < \phi^{-1}(-\log t)^{1/m}) \d r \\
		&= \omega_d \left[ \frac{1}{k} r^k \right]_0^{\phi^{-1}(-\log t)^{1/m}} 
		= \frac{\omega_d}{k} \phi^{-1}(-\log t)^{k/m} .
	\end{align*}
	From this it is easy to see that $\ell_{\varrho_{k,0},\varrho_{k,1}}$ is continuous and strictly decreasing on its support, so only condition \eqref{Con:Lambda_k_iii} of Definition \ref{Def:Lambda_k} remains to be verified in order to enable application of Theorem \ref{Thm:gap_estimate}. To this end, observe that the above computation yields
	\begin{equation*}
		\ell_{\varrho_{k,0},\varrho_{k,1}}(\exp(-s))^{m/k}
		= \left( \frac{\omega_d}{k} \phi^{-1}(s)^{k/m} \right)^{m/k}
		= \left( \frac{\omega_d}{k} \right)^{m/k} \phi^{-1}(s) ,
	\end{equation*}
	which is concave in $s$ by the assumptions on $\phi$ and Lemma \ref{Lem:convex_inverse}. By Proposition \ref{Prop:alternative_cond}, this shows $\ell_{\varrho_{k,0},\varrho_{k,1}}$ to satisfy condition \eqref{Con:Lambda_k_iii} of Definition \ref{Def:Lambda_k} with parameter $k/m$. Therefore Theorem \ref{Thm:gap_estimate} applies and yields
	\begin{equation*}
		\gap_{\pi}(P) 
		\geq \frac{1}{k/m + 1} 
		= \frac{m}{k + m}_{\textbf{.}}
		\qedhere 
	\end{equation*}
\end{proof}

\newpage 
\begin{discussion} \
	\begin{enumerate}[(a)]
		\item We point out that there is a very intuitive relationship between the result of Theorem \ref{Thm:k-PSS_rot_inv} and our Wasserstein contraction result, Theorem \ref{Thm:contraction}: If we set $m=1$ in Theorem \ref{Thm:k-PSS_rot_inv}, we obtain $\pi \in \D_k(\R^d)$, so that Theorem \ref{Thm:contraction} applies and yields that $k$-PSS for $\pi$ is Wasserstein contractive with rate $\rho = k/(k+1)$. By \eqref{Eq:contraction_to_gap}, this in turn yields a spectral gap estimate of $1/(k+1)$, which is the same as that provided by Theorem \ref{Thm:k-PSS_rot_inv} (due to our setting $m=1$). Hence, if one interprets the contraction result purely as a spectral gap estimate, then Theorem \ref{Thm:k-PSS_rot_inv} generalizes it by introducing the parameter $m \in \R_{>0}$.
		\item Furthermore, in the cases $m > 1$, the value of $m$ intuitively provides information about the `degree of log-concavity' of $\varrho_{k,1}$, which controls the target density's rate of decay along rays emanating from the origin. This additional information enables a more sensitive estimation of the spectral gap, leading to a lower bound of $m/(k+m)$, which is strictly larger than the lower bound of $1/(k+1)$ that could be obtained without it.
		\item The cases $m < 1$ complement this by requiring \textit{less} information, as the resulting $\exp$-term in $\varrho$ does no longer need to be log-concave (e.g.~for $\phi(r) = r$, $m = 1/2$ we get $\phi(r^m) = \sqrt{r}$, which is strictly concave rather than convex). Correspondingly, the resulting lower bounds of $m/(k+m)$ are strictly smaller than the lower bound of $1/(k+1)$ one could obtain in the case $m=1$.
		\item Though we are unable to formally show that the spectral gap estimates provided by Theorem \ref{Thm:k-PSS_rot_inv} are sharp, there is substantial numerical evidence suggesting that they are. We refer to Appendix \ref{App:num_supp} for details.
		\item We deem it worthwhile to consider the implications of Theorem \ref{Thm:k-PSS_rot_inv} for the spectral gap of USS, which thus far was most effectively examined in \cite{SlavaUSS}. The theorem yields for USS applied to
		\begin{equation*}
			\varrho(r \theta) = \exp(-\phi(r^m)) \ind_{\ooint{0}{\kappa}}(r^m)
		\end{equation*}
		the estimate
		\begin{equation*}
			\gap_{\pi}(P) \geq \frac{m}{d + m}_{\textbf{.}}
		\end{equation*}
		Despite being a restriction of the theorem, this can still be viewed as a vast generalization of the most closely related previous result: With $\kappa = \infty$, $\phi(r) = \alpha r$ and $m \leq d$ we obtain the setting of \cite[Example 3.15]{SlavaUSS}. Aside from generalizing this to potentially bounded support, non-linear $\phi$ and larger values of $m$, we also replace the gap estimate $1/(\lceil d/m \rceil + 1)$ shown in \cite{SlavaUSS} for this setting by the much nicer (and for almost all $m$ also larger) estimate $m/(d + m) = 1/(d/m + 1)$.
		\item Note that in the limit $m \ra \infty$, Theorem \ref{Thm:k-PSS_rot_inv} yields a spectral gap of $1$, corresponding to instantaneous convergence. We can give some intuition for this: Suppose we were using $k$-PSS for a target density $\varrho$ of the form
		\begin{equation*}
			\varrho(x)
			= \norm{x}^{k-d} \, c\,  \ind_A(x)
		\end{equation*}
		for some $c > 0$ and $A \in \B(\R^d)$ with $\lambda_d(A) < \infty$. Then we would get for $t \in \ooint{0}{c}$ that
		\begin{equation*}
			\varrho_{k,1}(x) > t
			\quad \LRA \quad 
			c \ind_A(x) > t
			\quad \LRA \quad 
			x \in A
		\end{equation*}
		and thus
		\begin{equation*}
			\varrho_{k,0}(x) \I(\varrho_{k,1}(x) > t)
			= \norm{x}^{k-d} \ind_A(x) 
			\propto \varrho(x) ,
		\end{equation*}
		so that each $X$-update would sample directly from the target distribution. We say that the sampler is \textit{exact} in this setting. Furthermore, since for $r \geq 0$ one has
		\begin{equation*}
			\lim_{m \ra \infty} r^m
			= \begin{cases} 0 & r < 1 \\ 1 & r = 1 \\ \infty & r > 1\end{cases}
		\end{equation*}
		the density \eqref{Eq:k-PSS_rot_inv} converges for $m \ra \infty$ almost everywhere to
		\begin{equation*}
			r\theta \mapsto r^{k-d} \exp(-\inf \phi) \ind_{\ooint{0}{1}}(r) ,
		\end{equation*}
		which fits into the exact sampling setting described above via $c := \exp(-\inf \phi)$ and $A = \Bf(1) \setminus \{0\}$. In other words, $k$-PSS performs increasingly well for the target densities \eqref{Eq:k-PSS_rot_inv} with increasing $m$ because they become closer and closer approximations of a target density for which the sampler is exact. Note that the same dynamic is also at play in our next result, Theorem \ref{Thm:k-PSS_non_rot}.
	\end{enumerate}
\end{discussion}

\subsection{Rotationally Asymmetric Targets}

In our second application of Theorem \ref{Thm:gap_estimate}, we consider what happens to the spectral gap of $k$-PSS when the target density is not rotationally invariant. Here it turns out to be very helpful to first restrict the setting of Theorem \ref{Thm:k-PSS_rot_inv} by eliminating the domain parameter $\kappa$ and the functional parameter $\phi$, and then generalize it by introducing a different functional parameter $\chi$ that allows for rotational asymmetry. In the resulting setting, the asymmetry does not have any effect on the spectral gap estimates we obtain, i.e.~it does not appear to hamper the samplers' performance.

\begin{theorem} \label{Thm:k-PSS_non_rot}
	Let $k \in \R_{>0}$ and let $\pi \in \Pc(\R^d)$ be a distribution with density $\varrho$ given by
	\begin{equation*}
		\varrho(r \theta)
		= r^{k-d} \exp(- \chi(\theta) r^m) ,
		\qquad r \in \R_+, \theta \in \sph^{d-1} ,
	\end{equation*}
	where $\chi: \sph^{d-1} \ra \R_{>0}$ is a measurable function\footnote{Note that some implicit assumptions are made on $\chi$ by assuming $\varrho$ to be integrable. These assumptions are rather weak, however. For example, assuming $\chi$ to be uniformly lower-bounded straightforwardly yields the required integrability.} and $m \in \R_{>0}$. Then 
	\begin{equation*}
		\gap_{\pi}(P)
		\geq \frac{m}{k+m}_{\textbf{,}}
	\end{equation*}
	where $P$ is the transition kernel of $k$-PSS for $\pi$.
\end{theorem}
\begin{proof}
	Note firstly that
	\begin{equation*}
		\varrho_{k,1}(r \theta) 
		= \exp(- \chi(\theta) r^m) ,
		\qquad r \in \R_+, \theta \in \sph^{d-1} ,
	\end{equation*}
	and secondly that by $\chi > 0$ we have $\varrho_{k,1} \leq \varrho_{k,1}(0) = \exp(0) = 1$. For $t \in \ooint{0}{1}$ one has
	\begin{align*}
		\varrho_{k,1}(r \theta) > t \quad 
		&\LRA \quad r^m < - \log(t) / \chi(\theta) \\
		&\LRA \quad r < (- \log(t) / \chi(\theta))^{1/m} ,
	\end{align*}
	where the first equivalence again relies on $\chi > 0$ and the second on $r \mapsto r^{1/m}$ being strictly increasing. Thus, by the polar coordinates formula (Proposition \ref{Prop:polar_coords}),
	\begin{align*}
		\ell_{\varrho_{k,0},\varrho_{k,1}}(t)
		&= \int_{\R^d} \norm{x}^{k-d} \I(\varrho_{k,1}(x) > t) \d x 
		= \int_{\sph^{d-1}} \int_0^{\infty} r^{k-1} \I(\varrho_{k,1}(r \theta) > t) \d r \sigma_d(\d \theta) \\ 
		&= \int_{\sph^{d-1}} \left[ \frac{1}{k} r^k \right]_0^{\left( - \log(t)/\chi(\theta) \right)^{1/m}} \sigma_d(\d \theta) 
		= \frac{1}{k} \int_{\sph^{d-1}} \left( - \log(t)/\chi(\theta) \right)^{k/m} \sigma_d(\d \theta) \\ 
		&= \frac{1}{k} \int_{\sph^{d-1}} \chi(\theta)^{-k/m} \sigma_d(\d \theta) \cdot (- \log(t))^{k/m} 
		=: C(k, \chi, m) \cdot (- \log(t))^{k/m} .
	\end{align*}
	From this it is clear that $\ell_{\varrho_{k,0},\varrho_{k,1}}$ is continuous and strictly decreasing on its support $\ooint{0}{1}$, so it only remains to verify condition \eqref{Con:Lambda_k_iii} of Definition \ref{Def:Lambda_k} in order to apply Theorem \ref{Thm:gap_estimate}. To this end, we observe that the above computation shows in particular
	\begin{equation*}
		\ell_{\varrho_{k,0},\varrho_{k,1}}(\exp(-s))^{m/k} = C(k,\chi,m)^{m/k} s ,
	\end{equation*}
	which is linear, and therefore concave, in $s$. Thus, again by Proposition \ref{Prop:alternative_cond}, we have shown $\ell_{\varrho_{k,0},\varrho_{k,1}} \in \Lambda_{k/m}$, so that Theorem \ref{Thm:gap_estimate} applies and yields
	\begin{equation*}
		\gap_{\pi}(P)
		\geq \frac{1}{k/m + 1}
		= \frac{m}{k + m}_{\textbf{.}}
	\end{equation*}
\end{proof}

\begin{discussion} \ \label{Dis:RotAsy}
	\begin{enumerate}[(a)]
		\item As far as we can tell, the result in Theorem \ref{Thm:k-PSS_non_rot} is not closely related to any result for slice sampling from any previously published work. It appears to be a major advance in understanding the spectral gap of slice sampling, particularly USS, in rotationally asymmetric settings, as we will elaborate in the following.
		\item We note, however, that there is some overlap between the classes of distributions for which Theorems \ref{Thm:k-PSS_rot_inv} and \ref{Thm:k-PSS_non_rot} apply. Namely, when setting $\chi(\theta) := c$ for some $c > 0$ and all $\theta \in \sph^{d-1}$, as well as $\kappa := \infty$, $\phi(r) := c \, r$, and choosing $k \in \R_{>0}$ and $m \in \R_{>0}$ the same in both theorems, the target densities they consider coincide. Notably also the spectral gap estimates they provide in these cases coincide.
		\item Aside from allowing us to prove the desired result, the computation of $\ell_{\varrho_{k,0},\varrho_{k,1}}$ in the proof of Theorem \ref{Thm:k-PSS_non_rot} also shows that
		\begin{equation*}
			\ell_{\varrho_{k,0},\varrho_{k,1}}(\exp(-s))^{1/p} = C(k,\chi,m)^{1/p} s^{k/(m p)}
		\end{equation*}
		for any $p > 0$. This is concave in $s$ if and only if $k / (p m) \leq 1$, i.e.~$p \geq k/m$. By Proposition \ref{Prop:alternative_cond}, this yields $\ell_{\varrho_{k,0},\varrho_{k,1}} \in \Lambda_p$ for $p \geq k/m$ and $\ell_{\varrho_{k,0},\varrho_{k,1}} \not\in \Lambda_p$ for $p < k/m$. As both $k$ and $m$ are allowed to take arbitrary values in $\R_{>0}$, this is an example of our generalization of the statement of Theorem \ref{Thm:gap_estimate} from a discrete sequence of classes $(\Lambda_k)_{k \in \N}$ (as in \cite{PSS_paper}) to a continuous spectrum $(\Lambda_k)_{k \in \R_{>0}}$ (as in our Theorem \ref{Thm:gap_estimate}) leading to a quantifiable improvement in the resulting spectral gap estimates.
		\item Again it is worthwhile to contemplate the theorem's restriction to USS, i.e.~the case $k = d$. Here the theorem yields for USS applied to
		\begin{equation*}
			\varrho(r \theta) 
			= \exp(- \chi(\theta) r^m)
		\end{equation*}
		the estimate
		\begin{equation}
			\gap_{\pi}(P)
			\geq \frac{m}{d + m}_{\textbf{.}}
			\label{Eq:USS_gap_est_non_rot}
		\end{equation}
		\item As it may not be obvious how the theorem could be applied in practice, we point out that, for USS, setting $m := 2$ and $\chi(\theta) := \frac{1}{2} \theta^T \Sigma^{-1} \theta$ for a symmetric and positive definite $\Sigma \in \R^{d \times d}$, one obtains
		\begin{equation*}
			\varrho(r \theta)
			= \exp(- \frac{1}{2} \theta^T \Sigma^{-1} \theta \cdot r^2)
			= \exp(- \frac{1}{2} (r\theta)^T \Sigma^{-1} (r \theta) )
			\propto \Nc_d(r \theta; 0, \Sigma) ,
		\end{equation*}
		so that $\pi = \Nc_d(0, \Sigma)$. The lower bound on the spectral gap proven in this case is $2/(d + 2)$. By location invariance of USS\footnote{That is, the fact that the transition mechanism of USS does not depend on where the target distribution's high probably regions are, only how they are located relative to one another.}, the same lower bound holds for the sampler applied to any multivariate normal target $\Nc_d(a,\Sigma)$. \label{Enu:USS_gap_Gaussians}
		\item Furthermore, by leaving $m \in \R_{>0}$ arbitrary, setting $\chi(\theta) := \frac{1}{2} \big( \theta^T \Sigma^{-1} \theta \big)^{m/2}$ and again using the location invariance of USS, one can cover a large class of generalized Gaussian distributions, namely all those with densities of the shape
		\begin{equation*}
			\varrho(x)
			= \exp(- \frac{1}{2} \big( (x - a)^T \Sigma^{-1} (x - a) \big)^{m/2} )
		\end{equation*}
		for an $a \in \R^d$ and a positive definite $\Sigma \in \R^{d \times d}$. The estimate obtained in these cases is \eqref{Eq:USS_gap_est_non_rot}.
		\item Regarding the theorem's restriction to PSS, we note that, much like in Theorem \ref{Thm:k-PSS_rot_inv}, it is unfortunate that we only obtain a result for target densities with a pole of order $(d-1)$ at the origin. Nevertheless, we find it noteworthy that the only previously available quantitative theoretical result regarding PSS for rotationally asymmetric targets was the broadly applicable but comparatively unsharp convergence estimate in \cite{PolarSS}, which additionally does not seem to imply a spectral gap. In other words, the restriction of Theorem \ref{Thm:k-PSS_non_rot} to PSS is, as far as we can tell, both the first result to provide ``realistic'' (i.e.~potentially sharp) estimates of the convergence rate (cf.~\eqref{Eq:gap_implied_convergence}) of PSS for a class of rotationally asymmetric targets, and the first to estimate the sampler's spectral gap in such a setting, which, as explained in the introduction, has many useful implications besides convergence.
	\end{enumerate}
\end{discussion}

\subsection{USS for multivariate $t$-distributions}

By combining the strategy of Theorem \ref{Thm:k-PSS_non_rot} with the results of Theorems \ref{Thm:USS_contraction_std_multiv_t}, Lemma \ref{Lem:gap_identity_ell} and \eqref{Eq:contraction_to_gap}, we can ultimately estimate the spectral gap of USS for virtually arbitrary multivariate $t$-distributions.

\begin{theorem} \label{Thm:USS_gap_multiv_t}
	Let $\pi \in \Pc(\R^d)$ with density $\varrho$ given by
	\begin{equation}
		\varrho(r \theta)
		= \left( 1 + \frac{1}{m} \chi(\theta) r^2 \right)^{-(d+m)/2} ,
		\qquad r \in \R_+, \theta \in \sph^{d-1} ,
		\label{Eq:multiv_t}
	\end{equation}
	where $m > 2$ and $\chi: \sph^{d-1} \ra \R_{>0}$ is measurable. Then
	\begin{equation}
		\gap_{\pi}(P_{\pi})
		\geq 1 - \frac{d(d+m)}{(d+1)(d+m-1)} ,
		\label{Eq:USS_gap_multiv_t}
	\end{equation}
	where $P_{\pi}$ denotes the transition kernel of USS for $\pi$.
\end{theorem}
\begin{proof}
	First of all, it is easily seen that the transition kernel of USS for any density $\eta$ coincides, for any $c > 0$, with that of slice sampling for $\zeta := c \cdot \eta$ with factorization $\zeta_0 := c \cdot \mathbf{1}$, $\zeta_1 := \eta$. Of course the transition kernels coinciding means that the respective spectral gaps must also coincide. Furthermore, we trivially get $c \cdot \ell_{\mathbf{1},\eta} \equiv \ell_{c \cdot \mathbf{1}, \eta} \equiv \ell_{\zeta_0, \zeta_1}$. Therefore, if $\ell_{\mathbf{1},\varrho} \equiv c \cdot \ell_{\mathbf{1},\eta}$, then $\ell_{\mathbf{1},\varrho} \equiv \ell_{\zeta_0,\zeta_1}$, so that by Lemma \ref{Lem:gap_identity_ell} the spectral gap of USS for $\varrho$ coincides with that of slice sampling for $\zeta$. Since this gap in turn coincides with that of USS for $\eta$, we ultimately obtain coinciding spectral gaps of USS for $\varrho$ and USS for $\eta$. For proving Theorem \ref{Thm:USS_gap_multiv_t}, the relevant consequence of all this is
	\begin{equation}
		\ell_{\mathbf{1}, \varrho} \propto \ell_{\mathbf{1}, \eta} 
		\quad \RA \quad 
		\gap_{\pi}(P_{\pi})
		= \gap_{\nu}(P_{\nu}) ,
		\label{Eq:gap_identity_USS}
	\end{equation}
	where $\nu$ denotes the distribution with yet-to-be-specified density $\eta$ and $P_{\nu}$ the transition kernel of USS for $\nu$.

	Next, note that for $r \in \R_+$, $\theta \in \sph^{d-1}$, $t \in \ooint{0}{1}$, we get
	\begin{align*}
		\varrho(r \theta) > t \quad 
		&\LRA \quad 1 + \frac{1}{m} \chi(\theta) r^2 < t^{-2/(d+m)} \\
		&\LRA \quad r < \sqrt{ \frac{m}{\chi(\theta)} (t^{-2/(d+m)} - 1) } .
	\end{align*}
	Thus, by the polar coordinates formula (Proposition \ref{Prop:polar_coords}), and omitting some very elementary integration steps already performed in detail in other proofs,
	\begin{align*}
		\ell_{\mathbf{1},\varrho}(t)
		&= \int_{\R^d} \I(\varrho(x) > t) \d x \\ 
		&= \int_{\sph^{d-1}} \int_0^{\infty} \I\left(r < \sqrt{ \frac{m}{\chi(\theta)} (t^{-2/(d+m)} - 1) }\right) r^{d-1} \d r \sigma_d(\d \theta) \\ 
		&= \frac{1}{d} \int_{\sph^{d-1}} \left( \frac{m}{\chi(\theta)} (t^{-2/(d+m)} - 1) \right)^{d/2} \sigma_d(\d \theta) \\ 
		&= \frac{m^{d/2}}{d} \int_{\sph^{d-1}} \chi(\theta)^{-d/2} \sigma_d(\d \theta) \cdot \left( t^{-2/(d+m)} - 1 \right)^{d/2} .
	\end{align*}
	This shows that $\ell_{\mathbf{1},\varrho} \propto \ell_{\mathbf{1},\eta}$, where $\eta$ is given by \eqref{Eq:multiv_t} with $\chi \equiv \mathbf{1}$, which is simply \eqref{Eq:std_multiv_t}. Hence, \eqref{Eq:gap_identity_USS} yields $\gap_{\pi}(P_{\pi}) = \gap_{\nu}(P_{\nu})$. Finally, by Theorem \ref{Thm:USS_contraction_std_multiv_t} and \eqref{Eq:contraction_to_gap} (which can be applied here because the constraint $m > 2$ ensures that $\pi$ has finite second moment), we get
	\begin{equation*}
		\gap_{\nu}(P_{\nu})
		\geq 1 - \frac{d(d+m)}{(d+1)(d+m-1)}_{\textbf{.}}
	\end{equation*}
	Combining the last two partial results yields the claim.
\end{proof}

\begin{discussion} \
	\begin{enumerate}[(a)]
		\item By choosing $\chi(\theta) := \theta^T \Sigma^{-1} \theta$ for a positive definite matrix $\Sigma \in \R^{d \times d}$, choosing any $a \in \R^d$ and again using the location invariance of USS (cf.~footnote in Discussion \ref{Dis:RotAsy} \eqref{Enu:USS_gap_Gaussians}), one obtains the gap estimate \eqref{Eq:USS_gap_multiv_t} for USS applied to the target density
		\begin{equation*}
			\varrho(x)
			= \left( 1 + \frac{1}{m} (x - a)^T \Sigma^{-1} (x - a) \right)^{-(d+m)/2} .
		\end{equation*}
		Setting aside their normalization, the densities of this form are precisely those of arbitrary multivariate $t$-distributions, except that we require $m > 2$, whereas the distribution is also well-defined for $0 < m \leq 2$.
		\item Both here and in our result on USS for arbitrary multivariate Gaussians $\Nc_d(a, \Sigma)$ (cf.~Discussion \ref{Dis:RotAsy} \eqref{Enu:USS_gap_Gaussians}), the matrix $\Sigma$ that causes the target density to be rotationally asymmetric does not affect the spectral gap estimate we obtain. Our analysis (specifically the computations of $\ell_{\mathbf{1},\varrho}$, combined with \eqref{Eq:gap_identity_USS}) shows that this is actually a strong point of our technique, as the spectral gap of USS in these scenarios is indeed unaffected by $\Sigma$.
	\end{enumerate}
\end{discussion}

\begin{acks}
	The author thanks Daniel Rudolf for many inspiring discussions on the topic and gratefully acknowledges support by the Carl Zeiss Stiftung within the program ``CZS Stiftungsprofessuren'' and the project ``Interactive Inference''. Moreover, the author is thankful for the support of the Deutsche Forschungsgemeinschaft (DFG) within project 432680300 -- Collaborative Research Center 1456 ``Mathematics of Experiment''.
\end{acks}

\appendix
\section{Auxiliary Results} \label{App:aux_res}

\subsection{Polar Coordinates}

We frequently make use of the following well-known result, which we refer to as the \textit{polar coordinates formula}.

\begin{proposition} \label{Prop:polar_coords}
	For any function $g: (\R^d,\B(\R^d)) \ra (\R,\B(\R))$ that is measurable and integrable, one has
	\begin{equation*}
		\int_{\R^d} g(x) \d x
		= \int_{\sph^{d-1}} \int_0^{\infty} g(r \theta) r^{d-1} \d r \sigma_d(\d \theta) ,
	\end{equation*}
	where $\sigma_d$ denotes the surface measure on $(\sph^{d-1},\B(\sph^{d-1}))$.
\end{proposition}
\begin{proof}
	See \cite[Theorem 15.13]{Schilling}.
\end{proof}

At some points, we require the following consequence of the polar coordinates formula, which, following \cite{GPSS}, we refer to as \textit{sampling in polar coordinates}.

\begin{corollary} \label{Cor:sampling_in_polar_coords}
	Let $\xi \in \Pc(\R^d)$ with density $h_{\xi}: \R^d \ra \R_+$, then by Proposition \ref{Prop:polar_coords} we have
	\begin{equation*}
		\xi(A)
		= \int_{\R^d} \ind_A(x) h_{\xi}(x) \d x 
		= \int_{\sph^{d-1}} \int_0^{\infty} \ind_A(r \theta) h_{\xi}(r \theta) r^{d-1} \d r \sigma_d(\d \theta) ,
		\qquad A \in \B(\R^d) .
	\end{equation*}
	Thus $X \sim \xi$ can be sampled in polar coordinates as $X := R \cdot \Theta$ by sampling $(R, \Theta)$ from the joint distribution with density
	\begin{equation*}
		(r, \theta) 
		\mapsto h_{\xi}(r \theta) r^{d-1} \ind_{\R_+}(r)
	\end{equation*}
	w.r.t.~$\lambda_1 \otimes \sigma_d$, where $\lambda_1$ denotes the Lebesgue measure on $(\R, \B(\R))$.
\end{corollary}

\subsection{Some Elementary Results}

\begin{lemma} \label{Lem:phi_inverse}
	Let $\phi: \oointv{0}{\kappa} \ra \R$ be as in Definition \ref{Def:D_k_def}. Denote the domain of $\phi$ by $D_{\phi} := \ooint{0}{\kappa}$ and the image by $I_{\phi} := \phi(D_{\phi})$. Then one has $I_{\phi} = \ooint{\inf \phi}{\sup \phi}$, where 
	\begin{align*}
		\inf \phi
		&:= \inf_{r \in D_{\phi}} \phi(r) 
		= \lim_{r \searrow 0} \phi(r) \\
		\sup \phi 
		&:= \sup_{r \in D_{\phi}} \phi(r)
		= \lim_{r \nearrow \kappa} \phi(r) .
	\end{align*}
	Furthermore, $\phi$ maps bijectively onto $I_{\phi}$ and its inverse $\phi^{-1}: I_{\phi} \ra D_{\phi}$ is again strictly increasing.
\end{lemma}
\begin{proof}
	Because $\phi$ is a convex function defined on an open interval, it is continuous, which shows that its image is again an open interval. That this interval must have the claimed boundaries is clear from the fact that $\phi$ is strictly increasing.
	
	Observe that the strict monotonicity of $\phi$ also guarantees its injectivity and note that every function is surjective onto its image. Hence $\phi$ maps bijectively $D_{\phi} \ra I_{\phi}$ and there exists an inverse function $\phi^{-1}: I_{\phi} \ra D_{\phi}$.
	
	Let now $s_1, s_2 \in I_{\phi}$ with $s_1 > s_2$. Set $r_i := \phi^{-1}(s_i)$, $i=1,2$, so that in particular $\phi(r_i) = s_i$. Then, since $\phi$ is strictly increasing and $\phi(r_1) = s_1 > s_2 = \phi(r_2)$ by assumption, we must have $r_1 > r_2$. This yields $\phi^{-1}(s_1) = r_1 > r_2 = \phi^{-1}(s_2)$, so $\phi^{-1}$ is again strictly increasing.
\end{proof}

\begin{figure}[tb]
	\centering 
	\includegraphics[scale=0.65]{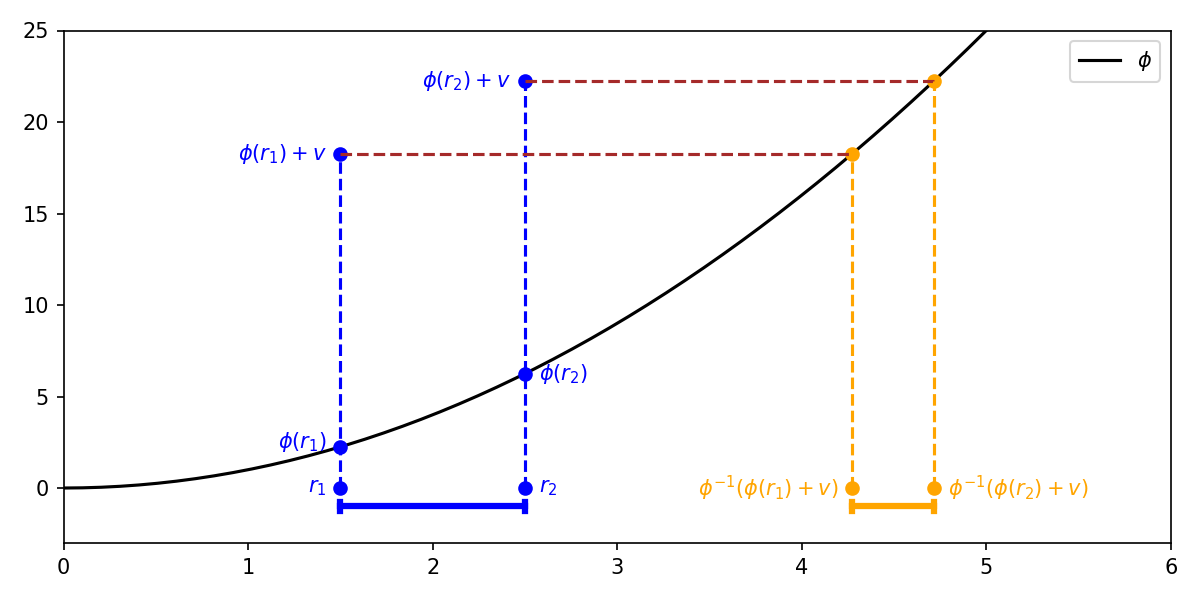}
	\caption{Illustration of the property shown in Lemma \ref{Lem:non_expansion_prop} for $\kappa=\infty$. Two points $r_1, r_2 \in \oointv{0}{\kappa}$ are fixed, $\phi$ is evaluated at each of them and the resulting function values are shifted upwards by the same value $v$, resulting in $\phi(r_1)+v$ and $\phi(r_2)+v$. The expressions $r_1^{\prime} := \phi^{-1}(\phi(r_1) + v)$ and $r_2^{\prime} := \phi^{-1}(\phi(r_2) + v)$ then correspond to the locations where $\phi$ reaches the shifted values. The lemma's claim is simply that the thick orange line, corresponding to $\abs{r_1^{\prime} - r_2^{\prime}}$, is at most as long as the thick blue line, corresponding to $\abs{r_1 - r_2}$.}
	\label{Fig:non_expansion}
\end{figure}

\begin{lemma} \label{Lem:non_expansion_prop}
	Let $\phi: D_{\phi} \ra I_{\phi}$ be as specified in Definition \ref{Def:D_k_def} and analyzed in Lemma \ref{Lem:phi_inverse}. Define the extension $\widehat{\phi}^{-1}: \R_{>0} \ra \R$ of its inverse $\phi^{-1}$ to $\R_{>0}$ by
	\begin{equation*}
		\widehat{\phi}^{-1}(s)
		:= \begin{cases} 
			\phi^{-1}(s) & s \in I_{\phi} = \oointv{\inf \phi}{\sup \phi} , \\ 
			\kappa & s \in \cointv{\sup \phi}{\infty} .
		\end{cases}
	\end{equation*}
	Note that assuming $\kappa = \infty$ leads to $\sup \phi = \infty$ (because a non-decreasing convex function on $\R_{>0}$ must either be constant or unbounded and we require $\phi$ to be strictly increasing, thus preventing it from being constant) and consequently $\widehat{\phi}^{-1} \equiv \phi^{-1}$. In particular, $\widehat{\phi}^{-1}$ cannot attain the value $\infty$.
	
	For this function $\widehat{\phi}^{-1}$ and all $r, \tilde{r} \in D_{\phi}$, $v \in \R_+$ one has
	\begin{align*}
		\abs{\widehat{\phi}^{-1}(\phi(r) + v) - \widehat{\phi}^{-1}(\phi(\tilde{r}) + v)}
		\leq \abs{r - \tilde{r}} .
	\end{align*}
	An illustration of this property can be found in Figure \ref{Fig:non_expansion}.
\end{lemma}
\begin{proof}
	We use arguments very similar to those in \cite{SlavaUSS}. 
	Fix $r, \tilde{r} \in D_{\phi}$ and $v \in \R_+$. Assume w.l.o.g.~that $r > \tilde{r}$ and set
	\begin{align*}
		r^{\prime} := \widehat{\phi}^{-1}(\phi(r) + v) , \\
		\tilde{r}^{\prime} := \widehat{\phi}^{-1}(\phi(\tilde{r}) + v) ,
	\end{align*}
	s.t.~what we need to prove becomes $\abs{r^{\prime} - \tilde{r}^{\prime}} \leq \abs{r - \tilde{r}}$.
	
	Since $\phi^{-1}$ is strictly increasing, $\widehat{\phi}^{-1}$ is still non-decreasing, so $r^{\prime} \geq \tilde{r}^{\prime}$ and
	\begin{align}
		&r^{\prime} = \widehat{\phi}^{-1}(\phi(r) + v) \geq \widehat{\phi}^{-1}(\phi(r)) = \phi^{-1}(\phi(r)) = r, \label{Eq:r_prime_ineq_1} \\
		&\tilde{r}^{\prime} = \widehat{\phi}^{-1}(\phi(\tilde{r}) + v) \geq \widehat{\phi}^{-1}(\phi(\tilde{r})) = \phi^{-1}(\phi(\tilde{r})) = \tilde{r}. \label{Eq:r_prime_ineq_2}
	\end{align}
	We now prove the claim by distinguishing four cases w.r.t.~the values of $r^{\prime}, \tilde{r}^{\prime}$ and $\sup \phi$, given that $r^{\prime} \geq \tilde{r}^{\prime}$ (as observed above) and $r^{\prime}, \tilde{r}^{\prime} \leq \kappa$ (by definition of $\widehat{\phi}^{-1}$):
	\begin{enumerate}
		\item Suppose $r^{\prime} = \tilde{r}^{\prime}$. Then $\abs{r^{\prime} - \tilde{r}^{\prime}} \leq \abs{r - \tilde{r}}$ holds trivially.
		\item Suppose $\tilde{r}^{\prime} < r^{\prime} < \kappa$. Then $r^{\prime}, \tilde{r}^{\prime} \in D_{\phi}$ and therefore $\phi(r^{\prime}) = \phi(r) + v$ as well as $\phi(\tilde{r}^{\prime}) = \phi(\tilde{r}) + v$. Furthermore, it is well-known that $\phi$ being convex implies that
		\begin{align*}
			S_{\phi}(r_1,r_2)
			&:= \frac{\phi(r_1) - \phi(r_2)}{r_1 - r_2}
		\end{align*}
		is non-decreasing in $r_1$ for fixed $r_2$ and vice-versa. Thus
		\begin{align*}
			\frac{\phi(r^{\prime}) - \phi(\tilde{r}^{\prime})}{r^{\prime} - \tilde{r}^{\prime}}
			&= S_{\phi}(r^{\prime}, \tilde{r}^{\prime}) 
			\stackrel{\eqref{Eq:r_prime_ineq_1}}{\geq} S_{\phi}(r, \tilde{r}^{\prime}) 
			\stackrel{\eqref{Eq:r_prime_ineq_2}}{\geq} S_{\phi}(r, \tilde{r}) 
			= \frac{\phi(r) - \phi(\tilde{r})}{r - \tilde{r}} \\
			&= \frac{(\phi(r) + v) - (\phi(\tilde{r}) + v)}{r - \tilde{r}} 
			= \frac{\phi(r^{\prime}) - \phi(\tilde{r}^{\prime})}{r - \tilde{r}}_{\textbf{,}}
		\end{align*}
		which, because all occurring numerators and denominators are positive, implies $r^{\prime} - \tilde{r}^{\prime} \leq r - \tilde{r}$. 
		For the general case, i.e.~without assuming $r > \tilde{r}$, this translates precisely to the desired relation $\abs{r^{\prime} - \tilde{r}^{\prime}} \leq \abs{r - \tilde{r}}$.
		\item Suppose $\tilde{r}^{\prime} < \kappa = r^{\prime}$ and $\sup \phi = \infty$. Then $I_{\phi} = \oointv{\inf \phi}{\infty}$, meaning that $\phi(r) + v$ and $\phi(\tilde{r}) + v$ are contained in $I_{\phi}$ and thus we again have $r^{\prime}, \tilde{r}^{\prime} \in D_{\phi}$. Since the analysis of the second case relied solely on this property, it applies to this case as well.
		\item Suppose $\tilde{r}^{\prime} < \kappa = r^{\prime}$ and $\sup \phi < \infty$. Note that $r^{\prime}$ attaining the value $\kappa$ must mean $\kappa < \infty$. Hence we may continuously extend $\phi$ to the domain $\ocintv{0}{\kappa} = D_{\phi} \cup \{\kappa\} \subset \R_{>0}$, denoting the extension $\widehat{\phi}: \ocintv{0}{\kappa} \ra \R$, i.e.~we set
		\begin{align*}
			\widehat{\phi}(\hat{r})
			&:= \begin{cases}
				\phi(\hat{r}) & \hat{r} \in D_{\phi} , \\
				\sup \phi & \hat{r} = \kappa .
			\end{cases}
		\end{align*}
		By construction, $\widehat{\phi}$ is still convex and inverts the restriction of $\widehat{\phi}^{-1}$ to $\ocintv{\inf \phi}{\sup \phi}$. Observe that $\tilde{r}^{\prime} < \kappa$ again yields $\tilde{r}^{\prime} \in D_{\phi}$ and thus $\phi(\tilde{r}^{\prime}) = \phi(\tilde{r}) + v$ and that, by definition of the involved quantities, $r^{\prime} = \kappa$ implies $\phi(r) + v \geq \sup \phi$.
		Equipped with all of these relations, we can employ the same approach as in the second case and obtain
		\begin{align*}
			\frac{\widehat{\phi}(r^{\prime}) - \widehat{\phi}(\tilde{r}^{\prime})}{r^{\prime} - \tilde{r}^{\prime}}
			&= S_{\widehat{\phi}}(r^{\prime},\tilde{r}^{\prime})
			\stackrel{\eqref{Eq:r_prime_ineq_1}}{\geq} S_{\widehat{\phi}}(r,\tilde{r}^{\prime})
			\stackrel{\eqref{Eq:r_prime_ineq_2}}{\geq} S_{\widehat{\phi}}(r,\tilde{r})
			= \frac{\widehat{\phi}(r) - \widehat{\phi}(\tilde{r})}{r - \tilde{r}} \\
			&= \frac{\phi(r) - \phi(\tilde{r})}{r - \tilde{r}}
			= \frac{(\phi(r) + v) - (\phi(\tilde{r}) + v)}{r - \tilde{r}} 
			= \frac{(\phi(r) + v) - \phi(\tilde{r}^{\prime})}{r - \tilde{r}} \\
			&\geq \frac{\sup \phi - \phi(\tilde{r}^{\prime})}{r - \tilde{r}} 
			= \frac{\widehat{\phi}(r^{\prime}) - \widehat{\phi}(\tilde{r}^{\prime})}{r - \tilde{r}}_,
		\end{align*}
		again showing $r^{\prime} - \tilde{r}^{\prime} \leq r - \tilde{r}$, so in the general case $\abs{r^{\prime} - \tilde{r}^{\prime}} \leq \abs{r - \tilde{r}}$.
	\end{enumerate}
\end{proof}

\subsection{Auxiliary Results for Proving Spectral Gap Estimates}

\begin{lemma} \label{Lem:convex_inverse}
	Suppose $h: I_1 \ra I_2$ is a strictly monotone, continuous function mapping between open intervals $I_1, I_2 \subseteq \R$. Then $h$ maps bijectively onto its image, w.l.o.g.~$I_2$, and its inverse $h^{-1}: I_2 \ra I_1$ has the same monotonicity property. Furthermore,
	\begin{itemize}
		\item if $h$ is increasing, then it is convex if and only if $h^{-1}$ is concave
		\item if $h$ is decreasing, then it is convex if and only if $h^{-1}$ is convex.
	\end{itemize}
\end{lemma}
\begin{proof}
	See \cite[Lemma A.4]{PSS_paper}
\end{proof}

\begin{proposition} \label{Prop:alternative_cond}
	For any $k \in \R_{>0}$, given all the other conditions, condition (iii) of Definition \ref{Def:Lambda_k} is equivalent to concavity of $s \mapsto \ell(\exp(-s))^{1/k}$ restricted to the domain
	\begin{equation*}
		D := \oointv{- \log \sup\{t \in \R_{>0} \mid \ell(t) > 0\}}{\infty} .
	\end{equation*}
\end{proposition}
\begin{proof}
	For $k \in \N$ this was shown in \cite[Proposition A.5]{PSS_paper}. As the proof only requires $k$ to be positive, it does not need to be modified to cover all the cases $k \in \R_{>0}$.
\end{proof}

\section{Numerical Support} \label{App:num_supp}

\begin{figure}[tb]
	\centering
	\includegraphics[scale=0.6]{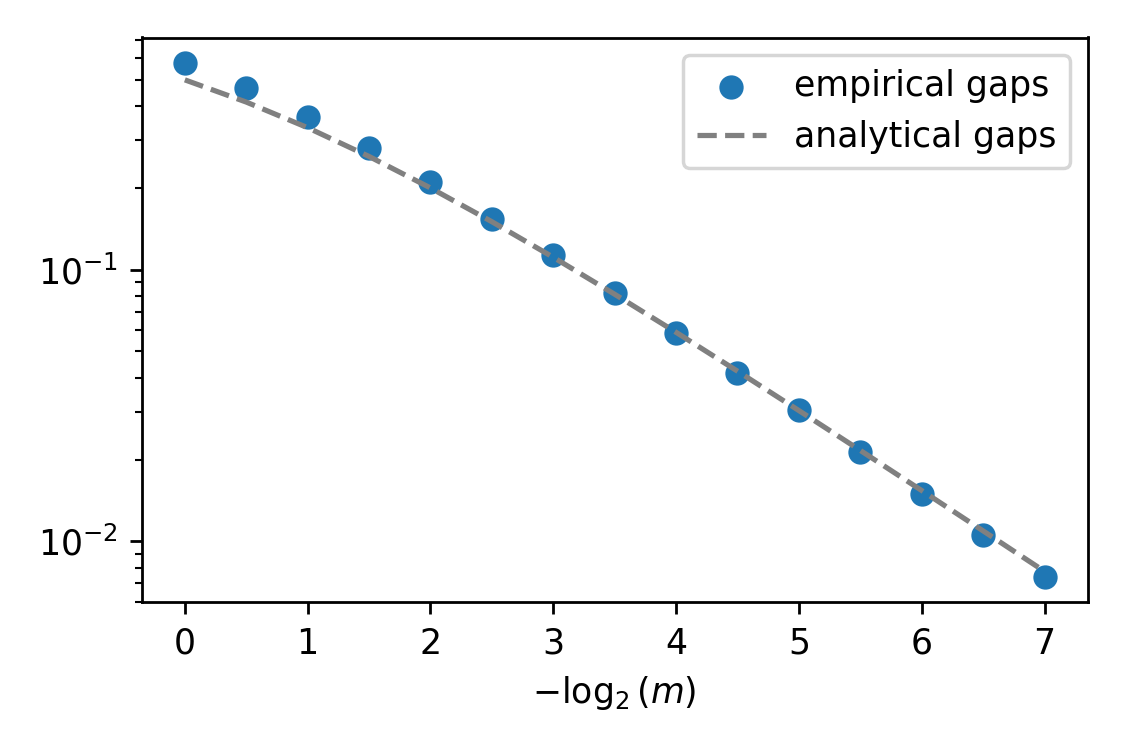}
	\includegraphics[scale=0.6]{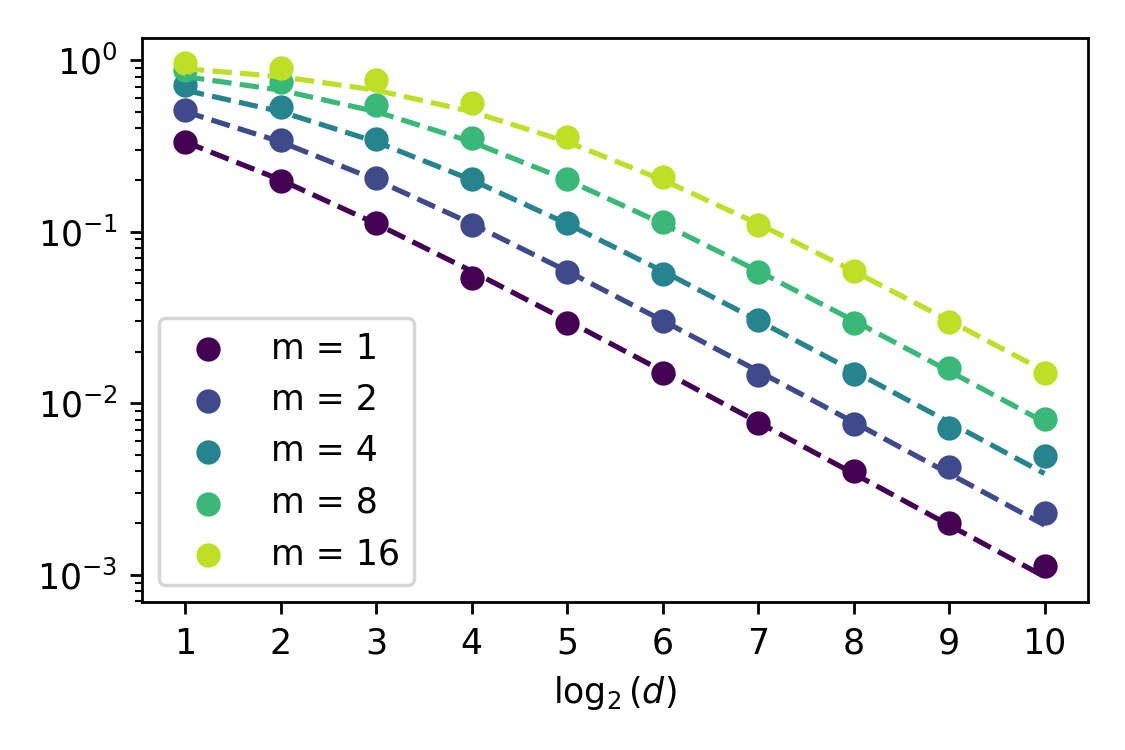}
	\caption{Comparisons of empirical spectral gaps (scattered points), computed according to \eqref{Eq:emp_gap_heuristic}, and spectral gap estimates provided by Theorem \ref{Thm:k-PSS_rot_inv} (dashed lines).
		The figure on the left shows these values for PSS applied to $\varrho(r \theta) = r^{1-d} \exp(-r^m)$ with parameter choices $m = 2^0, 2^{-0.5},..., 2^{-7}$. The IATs for the empirical gaps were computed from the sample log radii, i.e.~with $g(x) := \log \norm{x}$, which renders the results independent of the sample space dimension $d$.
		The right figure similarly shows empirical gaps and gap estimates for USS applied to $\varrho(r \theta) = \exp(-r^m)$ with parameter values $m = 2^0,...,2^4$ in dimensions $d=2^1,...,2^{10}$. Here the IATs for the empirical gaps were computed from the sample radii, i.e.~with $g(x) := \norm{x}$.}
	\label{Fig:emp_gaps}
\end{figure}

We can give some numerical support for the sharpness of our spectral gap estimates in Theorem \ref{Thm:k-PSS_rot_inv} by comparing them with heuristic approximations of the gaps from experimental data. The heuristic we consider is derived by identifying two different notions of effective sample sizes. 

Let $P$ be a transition kernel with invariant distribution $\pi$ and $\gap_{\pi}(P) > 0$. Denote by $(X_m)_{m \in \N_0}$ a Markov chain with transition kernel $P$ that is initialized with its invariant distribution, i.e.~$X_0 \sim \pi$.

Further denoting by $\IAT_g(P)$ the integrated autocorrelation time of $P$ w.r.t.~some reasonable summary function $g: \R^d \ra \R$, one would commonly define the effective sample size as
\begin{align*}
	n_{\text{eff}} := \frac{n}{\IAT_g(P)}_{\textbf{.}}
\end{align*}
On the other hand it is known from \cite[Corollary 3.37]{RudolfDiss} that, for a suitably normalized integrand, the mean squared Monte Carlo integration error of the truncated chain $(X_m)_{m \leq n}$ asymptotically (i.e.~as $n \ra \infty$) behaves exactly as
\begin{align*}
	\frac{2 - \gap_{\pi}(P)}{n \cdot \gap_{\pi}(P)}
	= \frac{1}{n} \left( \frac{2}{\gap_{\pi}(P)} -  1 \right) .
\end{align*}
As this quantity would be $1/n$ with i.i.d.~sampling, we may view its reciprocal as another notion of effective sample size. By informally identifying the two, we obtain the heuristic
\begin{align}
	\frac{\IAT_g(P)}{n} \approx \frac{1}{n} \left( \frac{2}{\gap_{\pi}(P)} -  1 \right)
	\quad \LRA \quad \gap_{\pi}(P) \approx \frac{2}{\IAT_g(P) + 1}
	\label{Eq:emp_gap_heuristic}
\end{align}
for the approximate computation of $\gap_{\pi}(P)$. This enables us to estimate spectral gaps from sample runs through their (also approximated) IATs.

The results of some example applications of this technique are shown in Figure \ref{Fig:emp_gaps}. They show the empirical gaps to match our gap estimates very well, which, if one believes the heuristic to be good, suggests the gap estimates to be very sharp. We note however that, in our experience, the approach only works this well if the spectral gaps in question are small -- if the proven gap estimates are large, the empirical gaps are consistently even larger, which causes discrepancies between the two.


\begin{thebibliography}{10}
	
	\bibitem{BesagGreen}
	Julian Besag and Peter~J. Green, \emph{Spatial statistics and {B}ayesian
		computation}, Journal of the Royal Statistical Society Series B \textbf{55}
	(1993), no.~1, 25--37.
	\MR{1210422}
	
	\bibitem{MCbook}
	Randal Douc, Eric Moulines, Pierre Priouret, and Philippe Soulier, \emph{Markov
		chains}, Springer Series in Operations Research and Financial Engineering,
	Springer, Cham, 2018.
	\MR{3889011}
	
	\bibitem{gapCLTasvar}
	James~M. Flegal and Galin~L. Jones, \emph{Batch means and spectral variance
		estimators in {M}arkov chain {M}onte {C}arlo}, The Annals of Statistics
	\textbf{38} (2010), no.~2, 1034--1070.
	\MR{2604704}
	
	\bibitem{SphericalSS}
	Michael Habeck, Mareike Hasenpflug, Shantanu Kodgirwar, and Daniel Rudolf,
	\emph{Geodesic slice sampling on the sphere}, arXiv preprint
	arXiv:2301.08056, 2023.
	
	\bibitem{gapCLT}
	Claude~P. Kipnis and Srinivasa R.~S. Varadhan, \emph{Central limit theorem for
		additive functionals of reversible {M}arkov processes and applications to
		simple exclusions}, Communications in Mathematical Physics \textbf{104}
	(1986), no.~1, 1--19.
	\MR{0834478}
	
	\bibitem{Latuszynski}
	Krzysztof Latuszyński and Daniel Rudolf, \emph{Convergence of hybrid slice
		sampling via spectral gap}, arXiv preprint arXiv:1409.2709, 2014.
	
	\bibitem{EllipticalSS}
	Iain Murray, Ryan~P. Adams, and David MacKay, \emph{Elliptical slice sampling},
	Proceedings of the 13th International Conference on Artificial Intelligence
	and Statistics, Proceedings of Machine Learning Research, vol.~9, 2010,
	pp.~541--548.
	
	\bibitem{SlavaUSS}
	Viacheslav Natarovskii, Daniel Rudolf, and Björn Sprungk, \emph{Quantitative
		spectral gap estimate and {W}asserstein contraction of simple slice
		sampling}, The Annals of Applied Probability \textbf{31} (2021), no.~2,
	806--825.
	\MR{4254496}
	
	\bibitem{SSNeal}
	Radford~M. Neal, \emph{Slice sampling}, The Annals of Statistics \textbf{31}
	(2003), no.~3, 705--767.
	\MR{1994729}
	
	\bibitem{Novak}
	Erich Novak and Daniel Rudolf, \emph{Computation of expectations by {M}arkov
		chain {M}onte {C}arlo methods}, Lecture Notes in Computational Science and
	Engineering \textbf{102} (2014).
	\MR{3329348}
	
	\bibitem{Ricci}
	Yann Ollivier, \emph{Ricci curvature of {M}arkov chains on metric spaces},
	Journal of Functional Analysis \textbf{256} (2009), no.~3, 810--864.
	\MR{2484937}
	
	\bibitem{SliceConv}
	Gareth~O. Roberts and Jeffrey~S. Rosenthal, \emph{Convergence of slice sampler
		{M}arkov chains}, Journal of the Royal Statistical Society Series B
	\textbf{61} (1999), no.~3, 643--660.
	\MR{1707866}
	
	\bibitem{PolarSS}
	\bysame, \emph{The polar slice sampler}, Stochastic Models \textbf{18} (2002),
	no.~2, 257--280.
	\MR{1904830}
	
	\bibitem{RudolfDiss}
	Daniel Rudolf, \emph{Explicit error bounds for {M}arkov chain {M}onte {C}arlo},
	Dissertationes Mathematicae \textbf{485} (2012), 1--93.
	\MR{2977521}
	
	\bibitem{PSS_paper}
	Daniel Rudolf and Philip Schär, \emph{Dimension-independent spectral gap of
		polar slice sampling}, arXiv preprint arXiv:2305.03685, 2023.
	
	\bibitem{Schilling}
	René~L. Schilling, \emph{Measures, integrals and martingales}, Cambridge
	University Press, 2005.
	\MR{2200059}
	
	\bibitem{GPSS}
	Philip Schär, Michael Habeck, and Daniel Rudolf, \emph{Gibbsian polar slice
		sampling}, Proceedings of the 40th International Conference on Machine
	Learning, Proceedings of Machine Learning Research, vol. 202, 2023,
	pp.~30204--30223.
	
	\bibitem{Villani}
	Cédric Villani, \emph{Topics in optimal transportation}, Graduate studies in
	mathematics, American Mathematical Society, 2003.
	\MR{1964483}
\end{thebibliography}

\providecommand{\bysame}{\leavevmode\hbox to3em{\hrulefill}\thinspace}
\providecommand{\MR}{\relax\ifhmode\unskip\space\fi MR }
\providecommand{\MRhref}[2]{%
	\href{http://www.ams.org/mathscinet-getitem?mr=#1}{#2}
}
\providecommand{\href}[2]{#2}

\end{document}